\documentclass[12pt]{amsart}

\usepackage{amsmath}
\usepackage{amsfonts}
\usepackage{amssymb}
\usepackage{amscd}
\usepackage{graphicx}
\RequirePackage[dvipsnames,usenames]{color}
\usepackage{soul,xcolor}
\setstcolor{red}
\usepackage{stmaryrd}
\SetSymbolFont{stmry}{bold}{U}{stmry}{m}{n}
\usepackage{mathtools}
\usepackage{booktabs}
\usepackage{multirow}\usepackage{subcaption}

\usepackage{mathdots}
\newtagform{tiny}{\tiny(}{)}

\usepackage{mathtools}
\usepackage{hyperref}
\hypersetup{colorlinks=true}
\usepackage[margin=1.25in]{geometry}

\usepackage{amsthm}
\usepackage{comment}
\usepackage[all,cmtip]{xy}
\usepackage{tikz-cd}
\usetikzlibrary{cd}

\usepackage[all]{xy}

\usepackage{etoolbox}
\apptocmd{\sloppy}{\hbadness 10000\relax}{}{}

\usepackage{enumitem}

\title{Extending one-forms on $F$-regular singularities}
\author{Tatsuro Kawakami \and Kenta Sato}
\address{Graduate School of Mathematical Sciences, University of Tokyo, 3-8-1 Komaba,
Meguro-ku, Tokyo 153-8914, Japan}
\email{kawakami@ms.u-tokyo.ac.jp}

\address{Department of Mathematics and Informatics, Chiba university, Chiba, 263-8522, Japan} 
\email{sato@math.s.chiba-u.ac.jp}

\def\ge{\geqslant}
\def\le{\leqslant}
\def\phi{\varphi}
\def\epsilon{\varepsilon}

\def\mapsto{\longmapsto}
\def\into{\lhook\joinrel\longrightarrow}
\def\onto{\relbar\joinrel\twoheadrightarrow}

\def\log{\operatorname{log}}
\def\Hom{\operatorname{Hom}}
\def\Spec{\operatorname{Spec}}

\def\Supp{\operatorname{Supp}}
\def\codim{\operatorname{codim}}

\def\Div{\operatorname{div}}
\def\Im{\operatorname{im}}
\def\Ker{\operatorname{ker}}

\def\Exc{\operatorname{Exc}}

\def\chara{\operatorname{char}}

\def\Sym{\operatorname{Sym}}

\def\res{\operatorname{res}}

\def\cod{\operatorname{cod}}

\def\restr{\operatorname{restr}}

\def\m{{\mathfrak m}}
\def\n{{\mathfrak n}}
\def\p{{\mathfrak p}}

\newcommand{\N}{\mathbb{N}}
\newcommand{\Q}{\mathbb{Q}}

\newcommand{\Z}{\mathbb{Z}}

\newcommand{\sO}{\mathcal{O}}

\newcommand{\bolda}{\boldsymbol{a}}
\newcommand{\boldb}{\boldsymbol{b}}
\newcommand{\boldx}{\boldsymbol{x}}

\newcommand{\cExt}{{\mathop{\mathcal{E}\! \mathit{xt}}}}

\newcommand{\choosing}[2]{\begin{pmatrix} #1 \\ #2 \end{pmatrix}}
\newcommand{\Map}{\mathrm{Map}}
\newcommand{\boldv}{\boldsymbol{v}}

\newcommand{\stacksproj}[1]{{\cite[\href{https://stacks.math.columbia.edu/tag/#1}{Tag~{#1}}]{stacks-project}}}

\newsavebox{\pullback}
\sbox\pullback{%
\begin{tikzpicture}%
\draw (0,0) -- (1ex,0ex);%
\draw (1ex,0ex) -- (1ex,1ex);%
\end{tikzpicture}}

\newsavebox{\pullbackdl}
\sbox\pullbackdl{%
\begin{tikzpicture}%
\draw (-1ex,0ex) -- (0ex,0ex);%
\draw (0ex,-1ex) -- (0ex,0ex);%
\end{tikzpicture}}

\newsavebox{\pushoutdr}
\sbox\pushoutdr{%
\begin{tikzpicture}%
\draw (-1ex,-1ex) -- (-1ex,0ex);%
\draw (-1ex,0ex) -- (0ex,0ex);%
\end{tikzpicture}}

\theoremstyle{plain}
\newtheorem{thm}{Theorem}[section] 
\newtheorem{cor}[thm]{Corollary}
\newtheorem{prop}[thm]{Proposition}

\newtheorem{deflem}[thm]{Definition-Lemma}
\newtheorem{lem}[thm]{Lemma}
\theoremstyle{definition} 
\newtheorem{defn}[thm]{Definition}

\theoremstyle{remark}
\newtheorem{rem}[thm]{Remark}

\newtheorem{defn and notation}[thm]{Definition and Notation}
\newtheorem{NOTATION}[thm]{Notation}
 
\newtheorem*{cl}{Claim}

\theoremstyle{plain}
\newtheorem{theo}{Theorem}

\numberwithin{equation}{thm}

\begin{document}
\tolerance = 9999

\begin{abstract}
We prove the logarithmic extension theorem for one-forms on strongly $F$-regular singularities. Additionally, we establish the logarithmic extension theorem for one-forms on three-dimensional klt singularities in characteristic $p>41$.
To this end, we reduce the problem to the logarithmic extension theorem for two-dimensional klt singularities with imperfect residue fields using a technique based on Cartier operators.
\end{abstract}

\maketitle
\markboth{Tatsuro Kawakami and Kenta Sato}{EXTENSION THEOREM FOR $F$-REGULAR SINGULARITIES}

\tableofcontents

\section{Introduction}

Let $X$ be a normal variety over a perfect field, and let $ f \colon Y\to X$ be a proper birational morphism with reduced $f$-exceptional divisor $E$. 
A fundamental question is whether every $i$-form on the smooth locus $U$ of $X$ extends to a logarithmic $i$-form on the simple normal crossing locus $V$ of $(Y,E)$.  
This condition is equivalent to the surjectivity of the restriction map:  
\[
f_{*}\Omega^{[i]}_Y(\log E)\coloneqq \iota_{*}\Omega^{i}_{V}(\log E|_V) \hookrightarrow \Omega^{[i]}_X\coloneqq j_{*}\Omega^{i}_{U},
\]  
or to the reflexivity of $f_{*}\Omega^{[i]}_Y(\log E)$, where $ j\colon U\hookrightarrow X$ and $\iota\colon V\hookrightarrow Y$ denote the inclusions.  
When this surjectivity holds for $X$, we say that $X$ satisfies the \textit{logarithmic extension theorem for $i$-forms}.  
If $X$ satisfies the stronger condition that the morphism
\[
f_{*}\Omega^{[i]}_Y\hookrightarrow \Omega^{[i]}_X,
\]  
is surjective, then we say that $X$ satisfies the \textit{regular extension theorem for $i$-forms}.

The study of extension theorems in characteristic zero has a rich history (see \cite{SvS85, Flenner88, Namikawa, GKK, GKKP, Graf21, KS21} for example).
Here, we focus on klt singularities, which play a crucial role in the minimal model program. For klt singularities, Greb--Kebekus--Kov\'{a}cs--Peternell \cite{GKKP} proved the regular extension theorem (see also \cite{KS21} for generalizations to rational singularities).

In positive characteristic, the logarithmic extension theorem fails for klt singularities even in dimension two \cite[Example 10.1]{Gra}. This motivates considering `strong $F$-regularity' instead, as strongly $F$-regular singularities provide a natural positive characteristic analog of klt singularities (see \cite{Hara98, Hara-Watanabe} for example).
Indeed, Graf \cite{Gra} proved that two-dimensional strongly $F$-regular singularities satisfy the logarithmic extension theorem (see also \cite[Theorem 4.7]{Kaw1}). 
In the proof, he observed the dual graph of minimal resolutions carefully, and therefore
generalizing his result to higher dimensions remains challenging due to the lack of a classification for higher-dimensional singularities.

Nonetheless, we establish the following result, which shows that strongly $F$-regular singularities satisfy the logarithmic extension theorem for one-forms in any dimension.

\begin{theo}[see Theorem \ref{thm:key} for a more general statement]\label{Introthm:F-regular}
    Let $X$ be a strongly $F$-regular variety over a perfect field of positive characteristic.
    Then $X$ satisfies the logarithmic extension theorem for one-forms, i.e., for any proper birational morphism $f\colon Y\to X$ with reduced $f$-exceptional divisor $E$, the natural restriction 
    \[
    f_{*}\Omega^{[1]}_Y(\log E)\hookrightarrow \Omega^{[1]}_X
    \]
    is surjective.
\end{theo}

In fact, our theorem is more general: it suffices to assume that 
(1) $X$ is strongly $F$-regular in codimension 2, 
(2) $X$ is quasi-$F$-injective in codimension 3 (see Section \ref{sec:qFinj} for details), and 
(3) $X$ satisfies Serre's condition $(S_4)$. We refer to Theorem \ref{thm:key} for further details.

Notably, the regular extension theorem does not hold in general in this setting. For instance, certain toric singularities in every characteristic $p>0$ violate the regular extension theorem (see \cite[Proposition 1.5]{Langer19} or \cite[Example 10.2]{Gra}).

The regular extension theorem for one-forms on klt singularities was originally proven by Greb--Kebekus--Kov\'{a}cs \cite{GKK} using Steenbrink-type vanishing, which relies on mixed Hodge theory.
Theorem \ref{Introthm:F-regular} provides a purely algebraic proof for the regular extension theorem for one-forms on klt singularities in characteristic zero, using their characteristic $p$-models (see Corollary \ref{cor:RET for klt in ch=0}).

\subsection{Klt threefolds in characteristic \texorpdfstring{$p>41$}{p>41}}
As discussed earlier, the klt assumption alone is insufficient to guarantee the logarithmic extension theorem in positive characteristic. However, we can still expect it to hold under additional assumptions on the characteristic $p>0$.
For instance, two-dimensional klt singularities in characteristic $p>5$ are
$F$-regular, and hence they satisfy the logarithmic extension theorem. Furthermore, \cite[Theorem E]{KTTWYY2} established that three-dimensional terminal singularities $X$ over a perfect field of characteristic
$p>41$ satisfy the logarithmic extension theorem for one-forms. The condition $p>41$ is used to ensure that 
$X$ is quasi-$F$-pure--a weaker notion of $F$-purity (see \cite[Section 9]{KTTWYY2} for further details).

In this paper, we also generalize \cite[Theorem E]{KTTWYY2} to three-dimensional klt singularities.
\begin{theo}\label{Introthm:klt}
    Let $X$ be a klt threefold over a perfect field of characteristic $p>41$.
    Then $X$ satisfies the logarithmic extension theorem for one-forms.
\end{theo}

\subsection{Lipman--Zariski type theorem}
In characteristic zero, it is conjectured that singularities with free tangent sheaves are smooth. This is known as the Lipman--Zariski Conjecture \cite{Lipman65}. The conjecture has been confirmed in numerous cases, and specifically for klt singularities, it was proven in \cite[Theorem 6.1]{GKKP} (see also \cite{Druel, Graf-Kovacs} for generalizations to log canonical singularities).

In contrast, in characteristic $p>0$, Lipman \cite{Lipman65} observed that rational double points (RDPs) of type 
$A_n$ violate the conjecture when $n+1$ is divisible by $p$. In these cases, the RDPs possess free tangent sheaves (see \cite{Graf(LZ-conjecture)} for further examples in positive characteristic). Notably, these singularities are all strongly 
$F$-regular, and even more specifically, they are toric.

As such, achieving smoothness appears unattainable in characteristic $p>0$, even under strong assumptions beyond the freeness of the tangent sheaves. Instead, we prove the following result: strongly 
$F$-regular singularities with free tangent sheaves are Frobenius liftable rather than smooth.

\begin{theo}\label{Introthm:LZ-conj}
    Let $X$ be a normal variety over a perfect field of characteristic $p>0$.
    Suppose that one of the following holds:
    \begin{enumerate}
        \item[\textup{(1)}] $X$ is strongly $F$-regular.
        \item[\textup{(2)}] $X$ is klt, $\dim\,X=3$, and $p>41$.
    \end{enumerate}
    If $T_X$ is locally free, then $X$ is locally Frobenius liftable.
\end{theo}

Here, we say that $X$ is \textit{Frobenius liftable} if $X$ lifts to $W_2(k)$ with its Frobenius morphism.
Furthermore, we say $X$ is \textit{locally Frobenius liftable} if it is Frobenius liftable Zariski locally.
It has become evident that Frobenius liftability imposes strong conditions on singularities (see \cite[Theorem B]{Kawakami-Witaszek} for example).

\subsection{Strategy of proofs}

All of Theorems \ref{Introthm:F-regular}, \ref{Introthm:klt}, and \ref{Introthm:LZ-conj} are reduced to the same problem: the surjectivity of 
the first reflexive Cartier operator:
\[
C \colon Z\Omega_X^{[1]}\to \Omega_X^{[1]}.
\]
We briefly recall reflexive Cartier operators (see Section \ref{subseq:Cartier} for details).
Let $j\colon U\hookrightarrow X$ be the inclusion of the smooth locus $U$ of $X$.
We define the locally free $\sO_U$-modules as follows:
 \begin{align*}
    &B\Omega^i_U\coloneqq \mathrm{im}(F_{*}d \colon F_{*}\Omega^{i-1}_U\to F_{*}\Omega^{i}_U)\,\,\text{and}\\
    &Z\Omega^{i}_U\coloneqq \mathrm{ker}(F_{*}d \colon F_{*}\Omega^{i}_U \to F_{*}\Omega^{i+1}_U)
\end{align*}
for all $i\geq 0$, where $F_{*}d$ denotes the pushforward of the differential map by the Frobenius morphism.
Using the Cartier isomorphism \cite[Theorem 1.3.4]{fbook}, we have the short exact sequence:
\[
0 \to B\Omega^{i}_U\to Z\Omega^i_U\xrightarrow{C} \Omega^{i}_U\to 0.
\]
By pushing forward via $j$, we obtain an exact sequence:
\begin{equation}\label{eq:reflexive Cartier exact sequence}
    0 \to B\Omega^{[i]}_X\to Z\Omega^{[i]}_X\xrightarrow{C} \Omega^{[i]}_X,
\end{equation}
where $C \colon Z\Omega_X^{[i]}\to \Omega_X^{[i]}$ is called \textit{the $i$-th reflexive Cartier operator}.
Note that, since pushforward is not generally right-exact, the reflexive Cartier operator is not always surjective.
If the reflexive Cartier operator is surjective, then we can show that $X$ satisfies the logarithmic extension theorem.

\begin{thm}[\textup{\cite[Theorem A]{Kaw4}}]\label{thm:Kaw4}
Let $X$ be a normal variety over a perfect field of positive characteristic. 
We fix an integer $i\geq 0$.
If the reflexive Cartier operator 
\[C \colon Z\Omega_X^{[i]}\to \Omega_X^{[i]}
\]
is surjective, then $X$ satisfies the logarithmic extension theorem for $i$-forms.
\end{thm}

By Theorem \ref{thm:Kaw4},
Theorems \ref{Introthm:F-regular} and \ref{Introthm:klt} are immediately reduced to the surjectivity of $C \colon Z\Omega_X^{[1]}\to \Omega_X^{[1]}$.
Furthermore, if $\Omega^{[1]}_X(\cong \mathcal{H}om_{\sO_X}(T_X,\sO_X))$ is locally free in addition, this map splits locally, implying that $X$ is locally Frobenius liftable \cite[Theorem 3.3]{Kawakami-Takamatsu}.
Consequently, Theorem \ref{Introthm:LZ-conj} is also established.

Thus, we focus on proving the surjectivity of $C \colon Z\Omega_X^{[1]}\to \Omega_X^{[1]}$.
For simplicity, we temporarily assume that $X$ is strongly $F$-regular.
By \eqref{eq:reflexive Cartier exact sequence}, we have a short exact sequence
\[
0\to B\Omega^{[1]}_X \to Z\Omega^{[1]}_X \to \mathrm{im}(C)\to 0
\]
and an inclusion $\mathrm{im}(C)\subset \Omega^{[1]}_X$.
Since $X$ is strongly $F$-regular, it can be shown that $B\Omega^{[1]}_X\cong F_{*}\sO_X/\sO_X$ satisfies the Serre condition $(S_3)$.
Therefore, if $\mathrm{im}(C)$ and $\Omega^{[1]}_X$ agree in codimension two, $\mathrm{im}(C)$ becomes reflexive, and thus coincides with $\Omega^{[1]}_X$.
This reduces the problem to proving the surjectivity of $C\colon Z\Omega^{[1]}_X \to \Omega^{[1]}_X$ in codimension two.

Consider a codimension two point $P\in X$, and set $(R,\m)=(\sO_{X,P},\m_P)$ and $S=\Spec\,R$.
The goal is to show the surjectivity of
\[
C\colon Z\Omega^{[1]}_S \to \Omega^{[1]}_S.
\]
As proven in \cite[Proposition 4.4]{Kaw4}, this surjectivity is equivalent to the logarithmic extension theorem for one-forms on $S$.
It is important to note that $S$ is a strongly $F$-regular surface singularity, although the residue field $R/\m$ is not perfect.

In the above discussion, we assumed that $X$ is strongly $F$-regular.
However, this condition can be weakened to the following:
$X$ is quasi-$F$-injective in codimension three, satisfies $(S_4)$, and is strongly $F$-regular in codimension two (see Theorem \ref{thm:key}). 
In particular, a variety $X$ in Theorem \ref{Introthm:klt} satisfies these conditions.

In conclusion, Theorems \ref{Introthm:F-regular}, \ref{Introthm:klt}, and \ref{Introthm:LZ-conj} 
are reduced to proving the logarithmic extension theorem for strongly $F$-regular surface singularities with imperfect residue fields.
Specifically, our main results are established by the following generalization of Graf's theorem \cite[Theorem 1.2]{Gra} to the case of an imperfect residue field:

\begin{theo}[\textup{Theorems \ref{mainthm:2-dim}} and \ref{mainthm:2-dim(lc)}]\label{Introthm:2-dim}
    Let $k$ be a perfect field of characteristic $p>0$.
    Let $(R,\m,\kappa\coloneqq R/\m)$ be a two-dimensional normal local domain essentially of finite type over $k$.
    Suppose that one of the following holds:
    \begin{enumerate}
        \item[\textup{(1)}] $S\coloneqq \Spec R$ is strongly $F$-regular. 
        \item[\textup{(2)}] $(S\coloneqq \Spec R,B)$ is lc and $p>5$, where $B$ is a reduced divisor on $S$.
    \end{enumerate}
    Then $(S,B)$ satisfies the logarithmic extension theorem for differential forms of any degree.
\end{theo}

From Theorem \ref{Introthm:2-dim}, we can prove the logarithmic extension theorem in codimension two for differential forms of any degree on strongly $F$-regular variety or lc varieties in characteristic $p>5$. 
(see Corollary \ref{cor:log ext for codim two}).

If $k\cong \kappa$, Theorem \ref{Introthm:2-dim} is already proven by Graf \cite{Gra}.
However, in our situation, the residue field $\kappa$ differs from the base field $k$, and many difficulties arise.
For instance, $\Omega^1_{S}$ is no longer of rank two, and $\Omega^i_S$ is non-zero even for $i\geq 3$.
To address these issue, we generalize many existing techniques to accommodate the case of imperfect fields. 
Among these advancements, we develop the concept of Campana pairs over imperfect fields (cf.~\cite[Problem 15.13]{Kebekus-Rousseau}).

Another key ingredient in the proof of Theorem \ref{Introthm:2-dim} is the geometric reducedness of exceptional curves in the minimal resolution.  
If we assume (2) in Theorem \ref{Introthm:2-dim}, this property is essentially proved in \cite{Sato23}.  
In contrast, in case (1), establishing geometric reducedness is more subtle, particularly in low characteristics.  
In \cite{Kawakami-Sato2}, we address this issue and prove that every exceptional curve in the minimal resolution of a strongly $F$-regular surface singularity is geometrically reduced over $\kappa$ (see \cite[Theorem A]{Kawakami-Sato2}).

\section{Preliminaries}

\subsection{Notation and terminology}
Throughout this paper, we use the following notation:
\begin{itemize}
\item A \textit{variety} is an integral separated scheme of finite type over a field. A curve (resp.~surface) is a variety of dimension one (resp.~two).
\item We say that an $\mathbb{F}_p$-scheme $X$ is {\em $F$-finite} if 
$F \colon X \to X$ is a finite morphism, and  we say that an $\mathbb{F}_p$-algebra $R$ is {\em $F$-finite} if $\Spec R$ is $F$-finite. If $R$ is a Noetherian $F$-finite $\mathbb{F}_p$-algebra, it is excellent, admits a dualizing complex, and $\dim\,R<\infty$ (cf.~\cite[Remark 13.6]{Gab04}).
\item Given a proper birational morphism $f\colon Y\to X$, we denote by $\Exc(f)$ the reduced exceptional divisor.
\item Given a Noetherian normal irreducible scheme $X$, a projective birational morphism $f \colon Y \to X$ is called \emph{a log resolution of $X$} if $Y$ is regular and $\mathrm{Exc}(f)$ is a simple normal crossing divisor (snc for short).
\item For the definition of the singularities of pairs appearing in the MMP (such as \emph{klt, plt, lc}), we refer to \cite[Definition 2.8]{Kol13}. Note that we always assume that the boundary divisor is effective although \cite[Definition 2.8]{Kol13} does not impose this assumption.
\item Given a normal variety $X$ over a perfect field and a reduced divisor $D$ on $X$, we denote $j_{*}(\Omega_{U}^{i}(\log D|_U)$ by $\Omega_X^{[i]}(\log D)$, where $j\colon U\hookrightarrow X$ is the inclusion from the snc locus $U$ of $(X, D)$.
\end{itemize}

\subsection{\texorpdfstring{$F$}{F}-splitting}

\begin{defn}\label{def:F-pure}
Let $X$ be an $F$-finite scheme.
We say $X$ is \textit{$F$-split} if the Frobenius map
\[
\sO_X\to F_{*}\sO_X
\] 
splits as an $\sO_{X}$-module homomorphism.
For a point $x\in X$,
we say that $X$ is \textit{$F$-pure at $x$} if $\Spec\,\sO_{X,x}$ is $F$-split.
We say that $X$ is \textit{$F$-pure} if $X$ is $F$-pure at every point of $X$.
\end{defn}

If $X$ is affine, then $X$ is $F$-split if and only if it is $F$-pure.

\begin{defn}\label{def:str F-reg}
Let $X$ be an $F$-finite scheme and $x\in X$ a point.
We say that $X$ is \textit{strongly $F$-regular at $x$} if, for every non-zero divisor $c\in\sO_{X,x}^{\circ}$, there exists an integer $e>0$ such that the map
\[
\sO_{X,x}\xrightarrow{F^e} F^{e}_{*}\sO_{X,x}\xrightarrow{\times F_{*}^{e}c}F^{e}_{*}\sO_{X,x}
\] splits as an $\sO_{X, x}$-module homomorphism.
We say that $X$ is \textit{strongly $F$-regular} if $X$ is strongly $F$-regular at every point of $X$.
\end{defn}

\begin{rem}\label{rem:SFR locus}
Let $X$ be an $F$-finite scheme.
Then the strongly $F$-regular locus is open in $X$ (\cite[Theorem 7.1 (5)]{LS-testideal} and \cite[Remark 13.6]{Gab04}).
\end{rem}

\begin{rem}\label{rem:F-regular}
    Let $X$ be a strongly $F$-regular scheme.
    Then it is normal and pseudo-rational \cite[Theorem 3.1]{Smith(F-rational)}.
    Moreover, if $X$ is $\mathbb{Q}$-Gorenstein, then it is klt \cite[Theorem 3.3]{Hara-Watanabe}.
    We further assume that $\dim\,X=2$. In this case, $X$ is rational since it is pseudo-rational, and in particular $\mathbb{Q}$-factorial.
    Therefore, two-dimensional $F$-finite strongly $F$-regular schemes are klt.
\end{rem}

\subsection{Quasi-\texorpdfstring{$F$}{F}-injectivity}\label{sec:qFinj}

Let $X$ be an $F$-finite scheme.
For a sheaf of Witt vectors, we have the following maps: 
\begin{alignat*}{4}
&(\textrm{Frobenius}) &&F \colon W_n\sO_X \to W_n\sO_X, \quad &&F(a_0, a_1, \ldots, a_{n-1}) &&\coloneqq (a_0^p,a_1^p, \ldots, a_{n-1}^p),\\ 
&(\textrm{Verschiebung}) \quad  && V \colon W_n\sO_X \to W_{n+1}\sO_X, \quad  &&V(a_0, a_1, \ldots, a_{n-1}) &&\coloneqq (0, a_0, a_1,\ldots,  a_{n-1}),\\
&(\textrm{Restriction}) && R \colon W_{n+1}\sO_X \to W_{n}\sO_X, \quad  &&R(a_0, a_1, \ldots, a_{n}) &&\coloneqq (a_0, a_1, \ldots,  a_{n-1}),
\end{alignat*}
where $F$ and $R$ are ring homomorphisms and $V$ is an additive homomorphism.
We denote the composition $W_n\sO_X \to \sO_X$ of restrictions by $R_{n-1}$.
We then have the following short exact sequence:
\begin{equation}\label{eq:WWO}
    0 \to F_*W_{n-1}\sO_X \xrightarrow{V} W_n \sO_X \xrightarrow{R_{n-1}} \sO_X \to 0
\end{equation}
of $W_n\sO_X$-modules.
We define a $W_n\sO_X$-module $\mathcal{B}_{n,X}$ by \[
\mathcal{B}_{n,X}\coloneqq\mathrm{coker}(F\colon W_n\sO_X\to F_{*}W_n\sO_X).\]
Then $\mathcal{B}_{n,X}$ has a natural $\sO_X$-module structure via $R_{n-1}\colon W_n\sO_X\to \sO_X$ \cite[Proposition 3.6 (3)]{KTTWYY1}.
We set $\mathcal{B}_{X}\coloneqq \mathcal{B}_{1,X}$.
Then the exact sequence \eqref{eq:WWO} induces the short exact sequence
\begin{equation}\label{eq:BBB}
    0\to F_{*}\mathcal{B}_{n-1,X} \to \mathcal{B}_{n,X} \xrightarrow{R_{n-1}} \mathcal{B}_{X} \to 0
\end{equation}
of $\sO_X$-modules.

Consider the following pushout diagram 
\begin{center}
\begin{tikzcd}
0 \arrow[r] & W_n\sO_X \arrow{r}{F} \arrow[d,"R_{n-1}"'] & F_* W_n\sO_X \arrow{d} \arrow[r] & \mathcal{B}_{n,X}\arrow[d,equal]\arrow[r] & 0 \\
0 \arrow[r] & \sO_X \arrow{r}{\Phi_{X, n}} & Q_{X,n} \arrow[lu, phantom, "\usebox\pushoutdr" , very near start, yshift=0em, xshift=0.6em, color=black] \arrow[r]& \mathcal{B}_{n,X}\arrow[r] & 0
\end{tikzcd}
\end{center}
of $W_n\sO_X$-modules.
Then $Q_{X,n}$ has a natural $\sO_X$-module structure, and the lower short exact sequence can be seen as a short exact sequence of $\sO_X$-module \cite[Proposition 2.9 (2)]{KTTWYY1}.
Moreover, we obtain the following commutative diagram:
\begin{equation}\label{diagram:Q Vs O}
    \begin{tikzcd}
0 \arrow[r] & \sO_X \arrow{r}{\Phi_{X, n}}\arrow[d,equal] & Q_{X,n} \arrow{r}\arrow[d] & \mathcal{B}_{n,X}\arrow[r]\arrow[d,"R_{n-1}"] & 0\\
0 \arrow[r] & \sO_X \arrow{r}{F} & F_{*}\sO_X \arrow{r} & \mathcal{B}_{X}\arrow[r] & 0.
\end{tikzcd}
\end{equation}

\begin{defn}\label{def:quasi-F-injective}
    Let $X$ be a normal $F$-finite Noetherian scheme.
    We say that $X$ is \textit{quasi-$F$-injective at a closed point $x\in X$} if there exists a positive integer $n\in\Z_{>0}$ such that
    the map
    \[
    H^j_{\m_x}(\sO_{X})\to H^j_{\m_x}(Q_{X,n})
    \]
    is injective for all $i\geq 0$.
    We say that $X$ is \textit{quasi-$F$-injective} if it is quasi-$F$-injective at every closed point.
    We say that $X$ is \emph{quasi-$F$-injective in codimension $r$} if there is a closed subset $Z$ of $X$ with $\cod(Z) >r$ such that the open subscheme $X \setminus Z$ is quasi-$F$-injective.
\end{defn}
\begin{rem}\label{rem:quasi-F-inj}
    We take $X$ as in Definition \ref{def:quasi-F-injective}.
    Noting that $Q_{X,1} = F_*\sO_X$, we can see that if $X$ is $F$-injective, then it is quasi-$F$-injective.
    
    Recall that a quasi-$F$-purity at a closed point $x\in X$ is equivalent to the splitting of the map
    $\sO_X \to Q_{X,n}$
    for some $n>0$ at $x\in X$ (see \cite[Proposition 2.8]{KTTWYY1}).
    Thus, if $X$ is quasi-$F$-pure, then it is quasi-$F$-injective.
    In summary, we have the following implication.
\[
\text{$F$-injective}\Longrightarrow
\text{quasi-$F$-injective}\Longleftarrow
\text{quasi-$F$-pure} 
 \]
\end{rem}

\begin{lem}\label{lem:qFinj}
    Let $X$ be a normal variety over an $F$-finite field.
    Suppose that $X$ is quasi-$F$-injective.
    Then there exists $n\in\Z_{>0}$ such that
    the map
    \[
    H^j_{\m_x}(\sO_{X,x})\to H^j_{\m_x}((Q_{X,n})_x)
    \]
    is injective for every $j$ and for every (not necessarily closed) point $x \in X$, where $\m_x \subseteq \sO_{X,x}$ is the maximal ideal.
\end{lem}
\begin{proof}
    We may assume that $X$ is affine.
    Since $X$ is $F$-finite, it follows from \cite{Gab04} that $X$ admits a dualizing complex $\omega_X^{\bullet}$.
    For every point $x \in X$, let $\delta_x$ be an integer such that $\omega_{X,x}^{\bullet}[-\delta_x]$ is a normalized dualizing complex of $\sO_{X,x}$ (cf.~\stacksproj{0A7W}).
    By the local duality \stacksproj{0AAK}, the map $H^j_{\m_x}(\sO_{X,x}) \to H^j_{\m_x}((Q_{X,n})_x)$ is injective if and only if the map 
    \[
    \cExt^{-j+\delta_x}(Q_{X,n},\omega^{\bullet}_X)\to \cExt^{-j+\delta_x}(\sO_X,\omega^{\bullet}_X).
    \]
    is surjective at $x$.

    Let $U_n$ be the set of (not necessarily closed) points $x \in X$ such that for every integer $l$, the map 
    $\cExt^{l}(Q_{X,n},\omega^{\bullet}_X)\to \cExt^{l}(\sO_X,\omega^{\bullet}_X)$
    is surjective at $x$.
    Noting that $\cExt^{l}(\sO_X,\omega^{\bullet}_X) = h^l(\omega_X^{\bullet})$ is zero for all but finitely many $l$, the subset $U_n$ is open in $X$.
    Moreover, it follows from the commutative diagram 
    \[
    \begin{tikzcd}
        \sO_X \arrow[rr,"\Phi_{X,n+1}"] \arrow[rrd,"\Phi_{X,n}"'] && Q_{X,n+1} \arrow[d] \\
        && Q_{X,n}
    \end{tikzcd}
    \]
    that we have $U_n \subseteq U_{n+1}$ for all $n$.
    Since $X$ is quasi-compact, the union $U \coloneqq \bigcup_n U_n$ is equal to $U_n$ for sufficiently large $n$.
    On the other hand, by the assumption that $U$ contains all closed points of $X$, we conclude that $U=X$, as desired.
\end{proof}

\begin{rem}\label{rem:qFinj2}
    By a similar argument to that in Lemma \ref{lem:qFinj}, we can show that $X$ is quasi-$F$-injective in codimension $r$ if and only if there exists an integer $n$ such that the map
    \[
    H^j_{\m_x}(\sO_{X,x})\to H^j_{\m_x}((Q_{X,n})_x)
    \]
    is injective for every $j$ and for every point $x \in X$ with $\dim \sO_{X,x} \le r$.
\end{rem}

\subsection{Cartier operator}\label{subseq:Cartier}
We fix a perfect field $k$ of positive characteristic.
Let $X$ be a smooth scheme over $k$.
We define $B\Omega_X^i$ and $Z\Omega_X^i$ to be the boundaries and the cycles of $F_*\Omega_X^{\bullet}$ at $i$: 
\begin{align} 
\label{eq:definition-of-B1} B \Omega_X^i &\coloneqq {\rm im}(F_*\Omega_X^{i-1} \xrightarrow{F_*d}  F_*\Omega_X^{i}),\\ 
Z \Omega_X^i &\coloneqq \Ker(F_*\Omega_X^{i} \xrightarrow{F_*d}  F_*\Omega_X^{i+1}) \nonumber.
\end{align}
By definition, we obtain the short exact sequence
\begin{equation}\label{eq:ZOmegaB}
    0\to Z\Omega^i_X \to F_{*}\Omega^i_X \to B\Omega^{i+1}_X \to 0.
\end{equation}
In particular, taking $i=0$, we obtain the short exact sequence
\begin{equation}\label{eq:OOB}
0\to \sO_X \to F_{*}\sO_X \to B\Omega^{1}_X \to 0.
\end{equation}
The Cartier isomorphism \cite[Theorem 1.3.4]{fbook} gives the following short exact sequence
\begin{equation} \label{eq:Cartier-for-Z1}
0 \to B\Omega^i_X \to Z\Omega^i_X \xrightarrow{C}  \Omega^i_X \to 0. 
\end{equation}

\subsubsection{Iterations}
Inductively on $n$, we will construct a locally free $\sO_X$-submodule $Z_n\Omega_X^i$ of $F^n_* \Omega_X$ which is contained in $F_* Z_{n-1}\Omega_X^i$, and an $\sO_X$-module homomorphism $C: Z_{n}\Omega^i_X \onto Z_{n-1}\Omega^i_X$ as follows:
We first set $Z_0 \Omega_X^i\coloneqq \Omega^i_X$ and $Z_1 \Omega_X^i\coloneqq Z\Omega^i_X$.
The morphism $C\colon Z_1\Omega_X \onto Z_0\Omega_X$ is induced by the Cartier isomorphism.
When $C: Z_{n}\Omega^i_X \onto Z_{n-1}\Omega^i_X$ is constructed, we define $Z_{n+1} \Omega^i_X$ and $C: Z_{n+1}\Omega_X^i \onto Z_n\Omega_X^i$ by the following pullback diagram:
\[
\begin{tikzcd}
Z_{n+1}\Omega^i_X \arrow[rd, phantom, "\usebox\pullback" , very near start, yshift=-0.3em, xshift=-0.6em, color=black] \arrow[d,"C"', twoheadrightarrow] \arrow[r, hook] & F_*Z_{n}\Omega^i_X \arrow[d, "F_*C"', twoheadrightarrow] \\
 Z_{n}\Omega^i_X \arrow[r, hook] & F_*Z_{n-1}\Omega^i_X.
\end{tikzcd}    
\]

We denote by $C_{n}\colon Z_n\Omega^i_X \onto \Omega^i_X$ the composite map
\[
Z_n\Omega^{i}_{X} \xrightarrow{C} Z_{n-1}\Omega^{i}_{X} \xrightarrow{C}\cdots \xrightarrow{C} \Omega^i_X.
\]
By induction on $n$ starting from \eqref{eq:ZOmegaB}, we have the following exact sequences
\begin{equation}\label{diagram for Z_n}
\begin{tikzcd}
0 \arrow{r} & Z_{n+1}\Omega^i_X \arrow[rd, phantom, "\usebox\pullback" , very near start, yshift=-0.3em, xshift=-0.6em, color=black] \arrow[d,"C"', twoheadrightarrow] \arrow[r] & F_*Z_{n}\Omega^i_X \arrow[d, "F_C"', twoheadrightarrow] \arrow[rr,"F_{*}d\circ F_{*}C_n"] && B\Omega^{i+1}_X \arrow[d,equal] \arrow[r] & 0\\
0 \arrow{r} & Z_{n}\Omega^i_X \arrow{r} & F_*Z_{n-1}\Omega^i_X \arrow[rr,"F_*d\circ F_{*}C_{n-1}"] && B\Omega^{i+1}_X \arrow{r} & 0.
\end{tikzcd}    
\end{equation}
Applying the snake lemma, we obtain the short exact sequence
\begin{equation}\label{eq:Z_ntoZ_{n-1}}
    0\to F^{n-1}_{*}B\Omega^i_X \to Z_{n+1}\Omega^i_X \xrightarrow{C} Z_{n}\Omega^i_X \to 0.
\end{equation}

We next define a locally free $\sO_X$-submodule $B_n \Omega_X^i$ of $F^n_* \Omega_X^i$, and an $\sO_X$-module homomorphism $C\colon B_n \Omega_X^i \onto B_{n-1}\Omega_X^i$ for all $n\geq 0$.
We first set $B_0\Omega_X^i(\log E)\coloneqq 0$ and $B_1 \Omega_X^i\coloneqq B\Omega_X^i$. For $n\geq 2$, we inductively define an $\sO_X$-module $B_n \Omega_X$ so that it fits in the following pull-back diagram:
\begin{equation}\label{diagram:construction for B_n}
\begin{tikzcd}
0 \arrow{r} & B_{n+1}\Omega^i_X \arrow[rd, phantom, "\usebox\pullback" , very near start, yshift=-0.3em, xshift=-0.6em, color=black] \arrow[d,"C"', twoheadrightarrow] \arrow{r} & Z_{n+1}\Omega^i_X \arrow[d,"C"', twoheadrightarrow] \arrow[r,"C_{n+1}"] & \Omega^{i}_X \arrow[d,equal] \arrow[r] & 0\\
0 \arrow{r} & B_{n}\Omega^i_X \arrow{r} & Z_{n}\Omega^i_X \arrow[r,"C_{n}"] & \Omega^{i}_X \arrow{r} & 0
\end{tikzcd}    
\end{equation}
In particular, we obtain a short exact sequence
\[
0 \to B_{n}\Omega^i_X \to Z_{n}\Omega^i_X \xrightarrow{C_{n}}  \Omega^{i}_X \to 0
\]
for all $n\geq 1$.
When $n=1$, this is nothing but \eqref{eq:Cartier-for-Z1}. 
By the snake lemma and \eqref{eq:Z_ntoZ_{n-1}}, we also obtain the short exact sequence
\begin{equation}\label{eq:B_ntoB_{n-1}}
    0\to F^{n-1}_{*}B\Omega^i_X \to B_{n+1}\Omega^i_X \xrightarrow{C} B_{n}\Omega^i_X \to 0.
\end{equation}
By \cite[Lemma 6.7]{KTTWYY2}, we also have the commutative diagram
\begin{equation}\label{eq:WOWOB}
    \begin{tikzcd}
0 \arrow{r} & W_n\sO_X \arrow[r,"F"]\arrow[d,"R_{n-1}"] & F_{*}W_n\sO_X \arrow[r,"s"]\arrow[d,"F_*R_{n-1}"] & B_n\Omega^{1}_X \arrow{r}\arrow[d,"C_{n-1}"] & 0\\
0 \arrow{r} &\sO_X \arrow[r,"F"] & F_{*}\sO_X \arrow[r,"d"] & B\Omega^{1}_X  \arrow[r] & 0,
\end{tikzcd} 
\end{equation}
where horizontal sequences are exact.

\subsubsection{Reflexive Cartier operators}

In what follows, we assume that $X$ is a normal variety over a perfect field $k$ of positive characteristic. 
Let $j\colon U\hookrightarrow X$ be the inclusion of the smooth locus.
We set 
\[
B_n \Omega_X^{[i]} \coloneqq j_{*}B_{n}\Omega_U^{i}\hspace{5mm}\text{and}\hspace{5mm} Z_n\Omega_X^{[i]} \coloneqq j_{*}Z_{n}\Omega_U^{i},
\]
which are reflexive $\sO_X$-modules.
Pushing forward the short exact sequence
\[
0 \to B_{n}\Omega^i_U \to Z_{n}\Omega^i_U \xrightarrow{C_{n}}  \Omega^{i}_U \to 0, 
\]
we obtain an exact sequence
\[
0 \to B_{n}\Omega^{[i]}_X \to Z_{n}\Omega^{[i]}_X \xrightarrow{C_{n}}  \Omega^{[i]}_X.
\]

\begin{lem}\label{lem:surjectiviety of C and C_n}
    Let $X$ be a normal variety over $k$.
    Then $C\colon Z\Omega^{[i]}_X \to \Omega^{[i]}_X$ is surjective if and only if $C_n\colon Z_n\Omega^{[i]}_X \to \Omega^{[i]}_X$ is surjective for every $n>0$.
\end{lem}
\begin{proof}
    Since we have the decomposition
    \[
    C_{n} \colon Z_n\Omega^{[i]}_{X} \xrightarrow{C} Z_{n-1}\Omega^{[i]}_{X} \xrightarrow{C} \cdots \xrightarrow{C} \Omega^{[i]}_X,
    \]
    the `if' part is obvious.
    Suppose that $C\colon Z\Omega^{[i]}_X \to \Omega^{[i]}_X$ is surjective.
    By the pullback diagram
    \begin{equation*}
\begin{tikzcd}
 Z_{n+1}\Omega^{[i]}_X \arrow[rd, phantom, "\usebox\pullback" , very near start, yshift=-0.3em, xshift=-0.6em, color=black] \arrow[d,"C"'] \arrow[r] & F_*Z_{n}\Omega^{[i]}_X \arrow{d}[swap]{F_*C} \\
 Z_{n}\Omega^{[i]}_X \arrow{r} & F_*Z_{n-1}\Omega^{[i]}_X ,
\end{tikzcd}    
\end{equation*}
we can inductively prove that $C\colon Z_n\Omega^{[i]}_X \to Z_{n-1}\Omega^{[i]}_X$ is surjective for every $n>0$.
Here, we note that left exact functors preserve the pullback diagrams in the category of $\sO_X$-modules.
Thus, the other direction holds.
\end{proof}

\begin{lem}\label{lem:B vs BOmega}
    Let $X$ be a normal variety over $k$.
    Then we have the commutative diagram:
    \begin{equation}\label{eq:Bvs BOmega}
        \begin{tikzcd}
        \mathcal{B}_{n,X}\arrow[r,hookrightarrow]\arrow[d,"R_{n-1}"'] & B_{n}\Omega^{[1]}_X\arrow[d,"C_{n-1}"]\\
        \mathcal{B}_{X}\arrow[r,hookrightarrow] & B\Omega^{[1]}_X\\
    \end{tikzcd}
    \end{equation}
    If $X$ is $F$-pure in addition, then the horizontal maps are isomorphisms.
\end{lem}
\begin{proof}
    Since $\mathcal{B}_X\cong F_{*}\sO_X/\sO_X$ is torsion-free, we can inductively deduce that $\mathcal{B}_{n,X}$ is torsion-free by \eqref{eq:BBB}.
    Thus, the reflexivization
    \[
    \mathcal{B}_{n,X}\hookrightarrow B_n\Omega^{1}_X
    \]
    is injective. Now, the first assertion follows from \eqref{eq:WOWOB}.
    
    To show the second assertion, suppose that $X$ is $F$-pure in addition. 
    We show that $\mathcal{B}_{n,X}$ is reflexive.
    Since $X$ is $F$-pure, the short exact sequence
    \[
    0\to \sO_X \to F_{*}\sO_X \to \mathcal{B}_{X} \to 0
    \]
    split locally, and thus, $\mathcal{B}_{X}$ is reflexive.
    By \eqref{eq:BBB} again, we conclude that $\mathcal{B}_n$ is reflexive inductively.
\end{proof}

\section{Extending one-forms}

In this section, we deduce Theorems \ref{Introthm:F-regular}, \ref{Introthm:klt}, and \ref{Introthm:LZ-conj} from Theorem \ref{Introthm:2-dim}.

\begin{lem}\label{lem:exactness in codim two}
    Let $V$ be a normal variety over a perfect field $k$ of positive characteristic.
    Suppose that
    \begin{enumerate}[label=\textup{(\arabic*)}]
        \item there exists a log resolution $f: W \to V$,
        \item $R^if_* \sO_W=0$ for $i\in\{1,2\}$, and
        \item $f_*\Omega_W^1(\log \Exc(f)) = \Omega_V^{[1]}$.
\end{enumerate}
    Then $C_n\colon Z_n \Omega_V^{[1]} \to \Omega_V^{[1]}$ is surjective for every $n \geq 1$.
    \end{lem}
\begin{proof}
    The surjectivity of $C_1$ follows from \cite[Proposition 4.4]{Kaw4}, and we obtain the surjectivity of $C_n$ for all $n\geq 1$ by Lemma \ref{lem:surjectiviety of C and C_n}.
\end{proof}

Thanks to Theorem \ref{Introthm:2-dim} and Lemma \ref{lem:exactness in codim two}, 
we prove the following theorem, which includes both \cite[Theorem D]{Kaw4} and \cite[Theorem 9.5]{KTTWYY2} as special cases.

\begin{thm}\label{thm:key}
    Let $X$ be a normal variety over a perfect field $k$ of positive characteristic.
    Suppose that the followings hold:
    \begin{enumerate}[label=\textup{(\arabic*)}]
        \item $X$ is strongly $F$-regular in codimension $2$.
        \item $X$ is quasi-$F$-injective in codimension $3$, and
        \item $X$ satisfies Serre's condition $(S_4)$, 
    \end{enumerate}
    Then $C\colon Z\Omega^{[1]}_X \to \Omega^{[1]}_X$ is surjective.
    In particular, $X$ satisfies the logarithmic extension theorem for one-forms.
\end{thm}
\begin{rem}
    Note that we can apply this theorem even when $\dim X\leq 3$. In this case, the assumption (3) is equivalent to saying that $X$ is Cohen--Macaulay.
\end{rem}
\begin{proof}
    Since the assertion is local on $X$, we may assume that $X$ is affine.
    Since there exists a log resolution of $\Spec \sO_{X,x}$ for every codimension two point $x \in X$, we can take an open subscheme $U \subseteq X$ such that $\codim_X(X \setminus U) \geq 3$ and a log resolution $f\colon W \to U$.
    By assumption (2), after shrinking $U$, we may also assume that $U$ is strongly $F$-regular (see Remark \ref{rem:SFR locus}).
    Combining the assumption (3) with Theorem \ref{Introthm:2-dim} and Remark \ref{rem:F-regular}, the conditions (2) and (3) of Lemma \ref{lem:exactness in codim two} are satisfied around every codimension two point of $U$.
    Therefore, after shrinking $U$ again, we may assume that $C_n \colon Z_n\Omega_U^{[1]} \to \Omega_U^{[1]}$ is surjective for every $n \ge 1$.
    
We set $Z\coloneqq X \setminus U$ and consider the following commutative diagram:
\[
\begin{tikzcd}
H^3_{Z}(\sO_X)\arrow[r, "\beta_n"]\arrow[d,"\alpha"] & H^3_{Z}(Q_{X,n})\arrow[d]\\
\bigoplus_{i=1}^l H^3_{\m_i}(\sO_{X,x_i})\arrow[r,"\bigoplus_{i} \beta_{n,i}"]& \bigoplus_{i=1}^l H^3_{\m_i}((Q_{X,n})_{x_i}),
\end{tikzcd}
\]
where $x_1, \dots, x_l$ are the generic points of $Z$ of height $3$ and $\m_i$ is the maximal ideal of $\sO_{X,x_i}$.
By Lemma \ref{lem:unmixed cohomogy} below, the localization map $\alpha$ is injective.
Since $X$ is $F$-injective in codimension $3$, the morphism $\beta_{n,i}$ is injective for every $i$ and sufficiently large $n$ (Remark \ref{rem:qFinj2}).
Therefore, we conclude that $\beta_n$ is also injective.

Then by \eqref{diagram:Q Vs O}, we have the following commutative diagram
\[
\begin{tikzcd}
H^2_{Z}(Q_{X,n})\arrow[r, twoheadrightarrow]\arrow[d] & H^2_{Z}(\mathcal{B}_{n,X})\arrow[d,"R_{n-1}"]\\
  \mathllap{0\,=\,\, } H^2_{Z}(F_{*}\sO_X)\arrow[r]& H^2_{Z}(\mathcal{B}_{X}),
\end{tikzcd}
\]
where the top horizontal arrow is surjective. Since $\codim_X(Z)\geq 3$ and $F_{*}\sO_X$ satisfies $(S_3)$, we have $H^2_{Z}(F_{*}\sO_X)=0$ \cite[Proposition 1.2.10 (a) and (e)]{BH93}.
Thus, $R^{n-1}\colon H^2_{Z}(\mathcal{B}_{n,X})\to H^2_{Z}(\mathcal{B}_X)$
is a zero map.

Since $X$ is affine, we have $H^1(U, \mathcal{B}_{m,U})\cong H^2_{Z}(\mathcal{B}_{m,X})$ for all $m>0$.
Moreover, by Lemma \ref{lem:B vs BOmega}, we have the natural isomorphisms
$\mathcal{B}_{m,U} \cong B_m\Omega_U^{[1]}$ for all $m>0$.
Then we obtain the following commutative diagram
\[
\begin{tikzcd}
H^2_{Z}(\mathcal{B}_{n,X})\arrow[r,"\cong"]\arrow[d,"R_{n-1}=0"'] & H^1(U,\mathcal{B}_{n,U})  \arrow[r,"\cong"] \arrow[d,"R_{n-1}"]& H^1(U,B_n\Omega^{[1]}_U)\arrow[d,"C_{n-1}"]  \\
H^2_{Z}(\mathcal{B}_{X})\arrow[r,"\cong"] & H^1(U,\mathcal{B}_{U}) \arrow[r,"\cong"] & H^1(U,B\Omega^{[1]}_U),
\end{tikzcd}
\]
where the commutativity follows from \eqref{eq:Bvs BOmega}.
Therefore, $C_{n-1}\colon H^1(U,B_n\Omega^{[1]}_U)\to H^1(U,B\Omega^{[1]}_U)$ is a zero map.

Since the reflexive Cartier operators are surjective on $U$, the diagram \eqref{diagram:construction for B_n} induces the following exact sequences:
\[
\begin{tikzcd}
0\arrow[r] & B_{n}\Omega_U^{[1]}\arrow[r]\arrow[d,"C_{n-1}"'] & Z_n\Omega_U^{[1]}\arrow[r,"C_n"]\arrow[d, "C_{n-1}"] & \Omega_U^{[1]}\arrow[d, equal] \arrow[r]& 0 \\
0\arrow[r] & B\Omega_U^{[1]}\arrow[r] & Z\Omega_U^{[1]}\arrow[r,"C"] & \Omega_U^{[1]}  \arrow[r]& 0.
\end{tikzcd}
\]
Pushing forward by $j\colon U\hookrightarrow X$, we obtain the following commutative diagram
\[
\begin{tikzcd}[column sep=3cm]
Z_{n}\Omega_X^{[1]}\arrow[r, "C_n"]\arrow[d, "C_{n-1}"'] & \Omega_X^{[1]}\ar[r]\arrow[d, equal] & H^1(U, B_{n}\Omega_U^{[1]})\arrow[d, "C_{n-1}=\,0"]\\
 Z\Omega_X^{[1]}\arrow[r, "C"] & \Omega_X^{[1]}\arrow[r] & H^1(U,B\Omega_U^{[1]}).
\end{tikzcd}
\]
Now, by diagram chasing, we obtain the surjectivity of $C\colon Z\Omega_X^{[1]}\to \Omega_X^{[1]}$, as desired.
The last assertion follows from Theorem \ref{thm:Kaw4}.
\end{proof}

\begin{lem}\label{lem:unmixed cohomogy}
    Let $r \ge 0$ be an integer. Let $A$ be a Noetherian integral domain which satisfies Serre's condition $(S_{r+1})$ and let $I \subseteq A$ be an ideal with $\mathrm{ht}(I) \ge r$.
    Suppose that $\p_1, \p_2 \dots, \p_l$ are the minimal prime ideals containing $I$ with height $r$.
    Then the following assertions hold:
    \begin{enumerate}[label=\textup{(\arabic*)}]
    \item The union $\bigcup_{i=1}^l \p_i$ is the set elements of $A$ which are zero-divisors on $H^r_I(A)$, in other words, the set of associated prime ideals of $H^r_I(A)$ is $\{\p_1,\dots,\p_l\}$. 
    \item The moprhism 
    \[
    \bigoplus_{i=1}^l \phi_i \colon H^r_I(A) \to  \bigoplus_{i=1}^l H^r_{\p_i A_{\p_i}}(A_{\p_i}) 
    \]
    is injective, where $\phi_i$ is the localization map
    \[
    H^r_I(A) \to H^r_I(A)_{\p_i} \cong H^r_{IA_{\p_i}}(A_{\p_i}) \cong H^r_{\p_iA_{\p_i}}(A_{\p_i}).
    \]
    \end{enumerate}
    \end{lem}

    \begin{proof}
        The assertion in (2) follows immediately from (1) (cf.~\stacksproj{0311}).
        For (1), we first note that 
        \[
        H^r_I(A)_\p \cong H^r_{IA_\p}(A_\p)
        \]
        for every prime ideal $\p \subseteq A$.
        In particular, it follows from the local duality that $H^r_I(A)_{\p_i}$ is non-zero for every $i$.
        This shows that $\p_i$ is a minimal prime ideal of $\Supp(H^r_I(A))$.
        It then follows from \stacksproj{05BV} that each $\p_i$ is an associated prime ideal of $H^r_I(A)$.
        By \stacksproj{00LD}, any element of $\bigcup_{i=1}^l \p_i$ is a zero-divisor of $H^r_I(A)$.

        For the converse inclusion, we take an element $s \in A \setminus \bigcup_{i=1}^l \p_i$ and prove that $s$ is a nonzero-divisor of $H^r_I(A)$.
        Since we have the exact sequence
        \[
        H^{r-1}_I(A/(s)) \to H^r_I(A) \xrightarrow{s} H^r_I(A),
        \]
        it suffices to show the vanishing $H^{r-1}_I(A/(s)) =0$.
        
        Take a prime ideal $\p \subseteq A$ containing $I$.
        If $\p$ does not contain $s$, then we have 
        \[
        \mathrm{depth} (A/(s))_{\p} = \mathrm{depth} (0) = \infty.
        \]
        If $\p$ contains $s$, then we have $\p \neq \p_i$ for every $i$.
        Therefore, one has $\mathrm{ht}(\p) \ge r+1$.
        Combining this with $(S_{r+1})$-condition, we have 
        \[
        \mathrm{depth} (A/(s))_{\p} = \mathrm{depth}(A_{\p}) -1 \ge r.
        \]
        It then follows from \cite[Proposition 1.2.10 (a)]{BH93} that one has 
        \[
        \mathrm{grade} (I, A/(s)) = \min_{\p \supseteq I} \mathrm{depth} (A/(s))_{\p} \ge r,
        \]
        which implies the desired vanishing $H^{r-1}_I(A/(s))=0$ \cite[Proposition 1.2.10 (e)]{BH93}.
    \end{proof}

\begin{thm}\label{thm:main}
    Let $X$ be a normal variety over a perfect field of characteristic $p>0$.
    Suppose that one of the following holds:
    \begin{enumerate}[label=\textup{(\arabic*)}]
        \item $X$ is strongly $F$-regular.
        \item $X$ is klt, $\dim\,X=3$, and $p>41$.
    \end{enumerate}
    Then the reflexive Cartier operator 
    \[
     C \colon Z\Omega_X^{[1]}\to \Omega_X^{[1]}
     \]
     is surjective.
\end{thm}
\begin{proof}
    We verify the conditions (1)--(3) in Theorem \ref{thm:key}.
    We first assume that $X$ is strongly $F$-regular. Then it is Cohen-Macaulay \cite[Proposition 3.9]{Takagi-Watanabe}, and (3) is satisfied.
    Conditions (1) and (2) can be immediately checked by definition.

    Next, we assume that $X$ is klt, $\dim\,X=3$, and $p>41$.
    By taking a $\Q$-factorialization of $X$ \cite[Theorem 2.14]{GNT16}, we may assume that $X$ is $\Q$-factorial.
    Then $X$ is quasi-$F$-split by \cite[Theorem A]{KTTWYY2}, and the condition (2) holds.
    The condition (3) can be confirmed by the fact that every three-dimensional klt singularity over a perfect field of characteristic $p>5$ are Cohen-Macaulay \cite[Corollary 1.3]{ABL}.
    \footnote{The conditions (2) and (3) can be also checked by \cite[Theorem B and Theorem D]{KTTWYY3}.}
    
    Finally, (1) can be verified from the fact that every two-dimensional klt singularity (with a possibly imperfect residue field) in characteristic $p>5$ is strongly $F$-regular (see \cite[Theorem 1.2]{Sato-Takagi(generalhyperplane)} or \cite[Theorem A]{Kawakami-Sato2}).
\end{proof}

\begin{proof}[Proof of Theorems \ref{Introthm:F-regular}, \ref{Introthm:klt}, and \ref{Introthm:LZ-conj}]
    Theorems \ref{Introthm:F-regular} and \ref{Introthm:klt} follow from Theorems \ref{thm:main} and \ref{thm:Kaw4}.
    We show Theorem \ref{Introthm:LZ-conj}.
    By Theorem \ref{thm:main}, we have the short exact sequence
    \[
    0\to B\Omega^{[1]}_X \to Z\Omega^{[1]}_X \xrightarrow{C} \Omega^{[1]}_X \to 0.
    \]
    Since $\Omega^{[1]}_X\cong \mathcal{H}om_{\sO_X}(T_X,\sO_X)$ is locally free, the above exact sequence splits locally, showing that $X$ is $F$-liftable by \cite[Theorem 3.3]{Kawakami-Takamatsu}.
\end{proof}

As a corollary of Theorem \ref{Introthm:F-regular}, we provide a purely algebraic proof of the regular extension theorem for one-forms on klt singularities in characteristic zero. This result was originally proven by Greb--Kebekus--Kov\'{a}cs \cite{GKK} using Steenbrink-type vanishing, which relies on mixed Hodge theory.

\begin{cor}\label{cor:RET for klt in ch=0}
    Let $X$ be a klt variety over a field $k$ of characteristic zero.
    Then $X$ satisfies the regular extension theorem for one-forms, i.e., for any proper birational morphism $f\colon Y\to X$, the natural restriction 
    \[
    f_{*}\Omega^{[1]}_Y\hookrightarrow \Omega^{[1]}_X
    \]
    is surjective.
\end{cor}
\begin{proof} 
    We may assume that $X=\Spec R$ is affine.
    We fix a log resolution $f\colon Y\to X$ which is projective and set $E\coloneq \Exc(f)$.
    We prove that $f_{*}\Omega^{[1]}_Y$ is reflexive.
    We take a finitely generated $\mathbb{Z}$-subalgebra $A$ of $k$, and take
    flat models $f_A\colon Y_A \to X_A=\Spec R_A$ of $f\colon Y\to X$ and a flat model $E_A$ of $E$.
    After enlarging $A$, we may assume that $f_A$ is proper birational with $E_A=\Exc(f_A)$.
    Moreover, we can also assume that $Y_A$ is smooth over $\Spec A$ and $E_A$ is a normal crossing divisor over $\Spec A$ in the sense of \cite[Assumption 8.9]{EV}.
    In particular, the sheaf $\Omega^1_{Y_A/A}(\log E_A)$ is defined as in \cite[Definition 8.10]{EV}.
    
    We shrink $\Spec A$ again so that for every point $s \in \Spec A$, the fiber $f_s\colon Y_s \to X_s$ of $f$ over $s$ is a log resolution of a normal variety $X_s$ with $E_s\coloneqq (E_A)\otimes_A \kappa(s) = \Exc(f_s)$, where $\kappa(s)$ is the residue field of $s$.
    Since we have $\Omega_{Y_A/A}^1(\log E_A)\otimes_A \kappa(s)  \cong \Omega_{Y_s/\kappa(s)}^1(\log E_s)$, it follows from \cite[Lemma 4.1]{Hara98} that we may assume that 
    \[
    (f_{A,*} \Omega_{Y_A/A}^1(\log E_A))\otimes_A \kappa(s) \cong f_{s,*} \Omega_{Y_s/\kappa(s)}^1(\log E_s)
    \]
    for every $s\in \Spec A$.
    Since $X$ is klt, by enlarging $A$, we may assume that the fiber $X_s$ over $s\in \Spec A$ is strongly $F$-regular for all closed points $s \in \Spec A$ by \cite[Theorem 5.2]{Hara98}.
     Then, by Theorem \ref{Introthm:F-regular}, 
    $f_{s,*}\Omega^{[1]}_{Y_s/\kappa(s)}(\log E_s)$
    is reflexive.
    \begin{cl}
        After enlarging $A$, we may assume that the reflexivization 
        \begin{equation}\label{eq:reduction}
            f_{A,*}\Omega^{i}_{X/A}(\log E_A)\hookrightarrow (f_{A,*}\Omega^{i}_{Y/A}(\log E_A))^{**}
        \end{equation}
        commutes with $\otimes_A \kappa(s)$.
        In particular, $f_{A,*}\Omega^{i}_{X/A}(\log E_A)\to (f_{A,*}\Omega^{i}_{Y/A}(\log E_A))^{**}$ is an isomorphism.
    \end{cl}
    \begin{proof}
        By enlarging $A$, we may assume that the cokernel of the map \eqref{eq:reduction} is a free $A$-module by the generic freeness.
        Then, the last assertion immediately follows from the first assertion.

        We show the commutativity with $\otimes_A \kappa(s)$ of \eqref{eq:reduction}.
        We denote an $R_A$-module $f_{A,*}\Omega^{i}_{X/A}(\log E_A)$ by $M_A$.
        It suffices to show that, after enlarging $A$,
        \[
        \Hom_{R_A}(M_A,R_A)\otimes_A \kappa(s)\cong \Hom_{R_s}(M_s,R_s),
        \]
        where $R_s\coloneqq R_A\otimes_A \kappa(s)$ and $M_s\coloneqq M_A\otimes_A \kappa(s)$.
        We take $m,n\in\Z_{>0}$ so that 
    \[
    R_A^{\oplus m}\to R_A^{\oplus n} \to M_A\to 0
    \]
    is exact.
    Then we obtain an exact sequence
    \[
    0\to \Hom_{R_A}(M,R_A)\to \Hom_{R_A}(R_A^{\oplus n},R_A) \to \Hom_{R_A}(R_A^{\oplus m},R_A).
    \]
    By enlarging $A$, all cokernels of the above maps are free over $A$, and by tensoring with $\kappa(s)$, we obtain an exact sequence
    \[
     0\to \Hom(M,R_A)\otimes_A \kappa(s)\to \Hom(R_A^{\oplus n},R_A)\otimes_A \kappa(s) \to \Hom (R_A^{\oplus m},R_A)\otimes_A \kappa(s).
    \]
    Since we have a natural isomorphism 
    \[
    \Hom_{R_A}(R_A^{\oplus l},R_A)\otimes_A \kappa(s)\cong \Hom_{R_s}(R_s^{\oplus l},R_s)
    \]
    for every $l\geq 0$, we obtain $\Hom_{R_A}(M,R_A)\otimes_A \kappa(s)\cong \Hom_{R_s}(M_s, R_s)$.
    \end{proof}
    Considering the restriction of $f_{A,*}\Omega^{i}_{X/A}(\log E_A)\cong (f_{A,*}\Omega^{i}_{Y/A}(\log E_A))^{**}$ to the generic fiber, we deduce that $f_{*}\Omega^{i}_{X}(\log E)$ is reflexive.
    By \cite[Theorem 3.1]{Graf-Kovacs}, we obtain $f_{*}\Omega^{[1]}_Y=f_{*}\Omega^{[1]}_Y(\log E)$, and we conclude.
\end{proof}

\section{Logarithmic differential sheaf}\label{Section:Logarithmic differential sheaf}

In this section, we study differential forms on schemes not necessarily of finite type over a field.
The main objects of interest will be the spectrum of local ring of a variety at a codimension two point, or a resolution of singularities of such a local ring.
We first remark that the dimension formula holds for such schemes as below.

\begin{rem}\label{rem:pair}
Let $R$ be a local ring essentially of finite type over a field $k$, and let
$X$ be an integral scheme of finite type over $R$.
Note that, for every point $x\in X$, a local ring $\sO_{X,x}$ is essentially of finite type over $k$, and it is catenary by \cite[Corollary 3.6]{HS06}.
Then, for every irreducible closed subscheme $Y \subseteq X$ containing $x\in X$, 
it follows from \cite[\href{https://stacks.math.columbia.edu/tag/02I6}{Tag 02I6} and \href{https://stacks.math.columbia.edu/tag/02IZ}{Tag 02IZ}]{stacks-project} that we have
\[
\dim(\sO_{X,x}) - \dim(\sO_{Y,x}) =\codim_X(Y).
\]
\end{rem}

\begin{defn}\label{defn:pair}
    Throughout Sections \ref{Section:Logarithmic differential sheaf} and \ref{Section:LET for surfaces}, 
    we say $(X,D)$ is a \emph{pair over a field $k$} if $X$ is a normal irreducible scheme essentially of finite type over $k$ (see \cite[Notation (13-c)]{KTTWYY1} for the definition), and $D$ is an effective $\Q$-divisor whose coefficients are less than or equal to one.
\end{defn}

\begin{defn}\label{defn:SNC}
A pair $(X,D=\sum_{i=1}^n c_i D_i)$ over a field $k$ has \emph{snc support} (resp.~\emph{geometrically snc support}) if 
the following holds:
\begin{enumerate}
    \item $X$ is regular (resp.~geometrically regular over $k$).
    \item For every sequence $1 \leq i_1 <i_2< \dots < i_m \leq n$, every connected component of 
$D_{i_1} \cap \dots \cap D_{i_m}$
is regular (resp.~geometrically regular over $k$) and has codimension $m$ whenever it is non-empty.
\end{enumerate}
\end{defn}

\begin{rem}
    By Remark \ref{rem:pair}, the following are equivalent to each other.
    \begin{enumerate}
        \item A pair $(X,D)$ over a field $k$ has snc support. 
        \item For every point $P \in X$ and local equations $x_1, \dots, x_\ell$ at $P$ of irreducible components $D_1,\ldots,D_\ell$ containing $P$, the sequence $x_1, \dots, x_\ell$ is a part of a regular system of parameters.
    \end{enumerate}
\end{rem}

\subsection{Definition of \texorpdfstring{$\mathcal{C}$}{C}-differentials}
In what follows, we generalize the notion of $\mathcal{C}$-differentials, which was introduced in \cite{Cam11} (see also \cite{JK,Kebekus-Rousseau}) to the case of positive characteristic.

\begin{lem}\label{lem:rsop vs diff basis}
  Let $(R,\m)$ be a local ring essentially of finite type over a field $k$, and let $x_1, \dots, x_r \in \m$ be a regular sequence such that $S\coloneqq R/(x_1, \dots, x_r)$ is geometrically regular over $k$.
  Then there exists a sequence $x_{r+1}, \dots, x_{N} \in R$ such that 
  \[
  \Omega^1_{R/k} = \bigoplus_{i=1}^N R dx_i.
  \]
\end{lem}
\begin{proof}
    Take a scheme $X$ of finite type over $k$, a closed subscheme $Z \subseteq X$, and a point $P \in Z$ such that we have
    \[
    R \cong \sO_{X,P} \textup{ and } S \cong \sO_{Z,P}.
    \]
    After shrinking $X$, we may assume that $X$ and $Z$ are geometrically regular over $k$.
    Then $X$ and $Z$ are smooth over $k$ \cite[\href{https://stacks.math.columbia.edu/tag/038X}{Tag 038X}]{stacks-project}, and $\Omega^1_{R/k}$ and $\Omega^1_{S/k}$ are free \cite[\href{https://stacks.math.columbia.edu/tag/02G1}{Tag 02G1}]{stacks-project}.
    Moreover, by localizing the second exact sequence \cite[\href{https://stacks.math.columbia.edu/tag/06AA}{Tag 06AA}]{stacks-project}, we obtain the short exact sequence
    \[
    0 \to I/I^2 \to \Omega^1_{R/k} \otimes_R S \to \Omega^1_{S/k} \to 0.
    \]
    Take a sequence $x_{r+1}, \dots, x_{N} \in R$ such that 
    \[
    d\overline{x}_{r+1}, \dots, d\overline{x}_{N} \in \Omega^1_{S/k}
    \] 
    is a basis.
    Since the second exact sequence splits, we have
    \[
    \Omega_{R/k}^1 \otimes_R S = \bigoplus_{i=1}^N S (dx_i \otimes 1).
    \]
    Combining this with the freeness of $\Omega_{R/k}^1$, we obtain the desired equation.
\end{proof}

\begin{NOTATION}\label{Notation:basis of Omega}
    Let $K/k$ be a field extension. Fix $x_1, \dots, x_N \in K$ such that $\Omega^1_{K/k} = \bigoplus_{i=1}^N K dx_i$.
    Then, for an integer $i \geq 0$, we have
    \[
    \Omega^i_{K/k} = \bigoplus_{\bolda} K d\boldx_{\bolda},
    \]
    where $\bolda$ runs through all elements of 
    \[
    \Sigma_i \coloneqq \choosing{[N]}{i} = \{\bolda=(a_1, \dots, a_i) \mid 1 \leq a_1 < \dots < a_i \leq N\}
    \]
    and we write
    \[
    d\boldx_{\bolda} \coloneqq dx_{a_1} \wedge \cdots \wedge dx_{a_i}.
    \]
    Moreover, for an integer $m \geq 0$, we have
    \begin{align}\label{eq:Sym}
    \Sym^m_K(\Omega^i_{K/k}) = \bigoplus_{A} K d\boldx^A,
    \end{align}
    where $A$ runs through all elements of 
    \[
    \Theta_{i,m} \coloneqq \Map(\Sigma_i, \N)_m = \left\{ A \colon \Sigma_i \to \N \ \middle| \sum_{\bolda \in \Sigma_i} A(\bolda) = m \right\}
    \]
    and we write $d\boldx^A \coloneqq \prod_{\bolda \in \Sigma_i} d\boldx_{\bolda}^{A(\bolda)}$.
    
    Finally, for every element $A \in \Theta_{i,m}$, we denote by $A(s)$ the order of $d\boldx^A$ at $dx_s$, that is,
    \[
    A(s) \coloneqq \sum_{\bolda \in \Sigma_i(s)} A(\bolda) \in \mathbb{Z}_{\geq0},
    \]
    where we write $\Sigma_i(s) \coloneqq \{\bolda=(a_1, \dots, a_i) \in \Sigma_i \mid a_k=s \textup{ for some $k$}\}$.
\end{NOTATION}

\begin{rem}\label{rem:Sym0}
In the above notation, if one has $i>N$, then we have
\[
\Sym^m_K(\Omega_{K/k}^i) = \left\{\begin{matrix} K & (\textup{if } m=0) \\ 0 &(\textup{if } m \neq 0). \end{matrix}\right.
\]
The equation \eqref{eq:Sym} still holds in this case since we have
\[
\Theta_{i,m}= \left\{\begin{matrix} \{\emptyset_{\Z_{\geq0}} \} & (\textup{if } m=0) \\ \emptyset &(\textup{if } m \neq 0) \end{matrix}\right.,
\]
where $\emptyset_{\Z_{\geq0}} \in \Z_{\geq0}^{\emptyset} = \mathrm{Map}(\emptyset, \Z_{\geq0})$ is the empty function.
\end{rem}

\begin{deflem}[\textup{cf.~\cite[Subsection 3.C]{JK}}]\label{deflem:differential module local}
    Let $(X, D=\sum_{i=1}^r c_iD_i)$ be a pair over a field $k$ with geometrically snc support.
    We further assume that $X=\Spec R$ for a local ring $R$.
    Let $K$ be the fraction field of $R$, and let $x_i \in R$ be an equation of $D_i$ for $1 \leq i \leq r$.
    We take a sequence $x_{r+1}, \dots, x_{N} \in R$ as in Lemma \ref{lem:rsop vs diff basis}.
    For integers $i,m \geq0$, we define the $R$-submodule $\Sym_{\mathcal{C}}^m \Omega^i_{R/k}(\log D)$ of $\Sym_K^m \Omega_{K/k}^i$ as 
    \[
        \Sym_{\mathcal{C}}^m \Omega^i_{R/k}(\log D) \coloneqq \bigoplus_{A \in \Theta_{i,m}} R \cdot \frac{d\boldsymbol{x}^A}{\prod_{s=1}^r x_s^{\lfloor c_s A(s) \rfloor}},
    \]
    where we adopt the notation in Notation \ref{Notation:basis of Omega}.
    This $R$-submodule is independent of the choice of $x_1,\dots, x_N$.
\end{deflem}

\begin{proof}
    Take another sequence $y_1, \ldots, y_N \in R$ satisfying the assumption.
    We fix $A \in \Theta_{i,m}$.
    It suffices to show that an element
    \[
    \omega \coloneqq \frac{d\boldsymbol{y}^A}{\prod_{s=1}^r y_s^{\lfloor c_s A(s) \rfloor}} \in \Sym^m_{K} \Omega^i_{K/k}
    \]
    is contained in 
    \[
    M\coloneqq\bigoplus_{A \in \Theta_{i,m}} R \cdot \frac{d\boldsymbol{x}^A}{\prod_{s=1}^r x_s^{\lfloor c_s A(s) \rfloor}}
    \]
    
    We first consider the case where $y_i = x_i$ for $i \geq2$.
    Take a unit $u \in R$ such that $y_1 = u x_1$ and we write $du = \sum_{i=1}^N r_i dx_i$ with $r_i \in R$.
    Then we have
    \[
    d\boldsymbol{y}_{\bolda} =\left\{  \begin{array}{lll} u d \boldx_{\bolda} +x_1 du \wedge d\boldx_{\check{\bolda}} &= (u+x_1r_1) d \boldx_{\bolda} + \sum_{i=2}^N x_1 r_i d\boldx_{\check{\bolda}_i} & \textup{(if $\bolda \in \Sigma_i(1)$)} \\ d \boldx_{\bolda} & &\textup{(otherwise)}\end{array} \right. ,
    \]
    where $\check{\bolda} \in \Sigma_{i-1}^N$ is the sequence obtained from $\bolda$ by removing $a_1=1$ and $\check{\bolda}_i$ is the sequence obtained from $\check{\bolda}$ by adding $i$.
    This shows that if we write
    \[
    d\boldsymbol{y}^A = \sum_{B \in \Theta_{i,m}} \alpha(B) d\boldx^B
    \]
    with $\alpha(B) \in K$, then every element $B$ with $\alpha(B)  \neq 0$ satisfies the following properties:
    \begin{itemize}
        \item $B(i) \geq A(i)$ for $i \neq 1$, and
        \item $\alpha(B) \in x_1^{A(1)-B(1)}R$.
    \end{itemize}
    Combining this with the assumption that $c_1 \leq1$, we have
    \[
    \frac{\alpha(B)}{y_1^{\lfloor c_1 A(1) \rfloor}} \in \frac{1}{x_1^{\lfloor c_1 B(1) \rfloor}} R, \textup{ and }    \frac{1}{y_i^{\lfloor c_i A(i) \rfloor}} \in \frac{1}{x_i^{\lfloor c_i B(i) \rfloor}} R \ (\forall i \neq 1).
    \]
    This proves $\omega \in M$, as desired.

    By the above argument, we may assume that $x_i= y_i$ for every $i=1, \dots, r$.
    In this case, if we write
    \[
    d\boldsymbol{y}^A = \sum_{B \in \Theta_{i,m}} \alpha(B) d\boldx^B
    \]
    with $\alpha(B) \in K$, then for every element $B$ with $\alpha(B)  \neq 0$, we have $B(i) \geq A(i)$ for $i =1, \dots, r$.
    This proves $\omega \in M$.
\end{proof}

\begin{rem}\label{rem:variant of C-diff}
    In the above notation, let $V$ be a finite dimensional $k$-module with basis $\{v_{j}\}_{j \in J}$.
    For integers $i,m \geq0$, we define the $R$-submodule $\Sym_{\mathcal{C}}^m (V \otimes_k \Omega^i_{R/k}(\log D))$ of $\Sym_K^m (V \otimes_k \Omega_{K/k}^i)$ as follows:
    We write
    \begin{align*}
    \Theta_{i,m}^{J} & \coloneqq \Map(J \times \Sigma_i , \Z_{\ge 0})_m \\
    & = \{ A \colon J \times \Sigma_i \to \Z_{\ge 0} \mid \sum_{(j,\bolda) \in J \times \Sigma_i} A(j,\bolda)=m\}
    \end{align*}
    and for every $A \in \Theta_{i,m}^{J}$, we set
    \[
    (\boldv \otimes d\boldx)^A \coloneqq \prod_{(j,\bolda) \in J \times \Sigma_{i}} (v_j \otimes d\boldx_{\bolda})^{A(j,\bolda)} \in \Sym_K^m (V \otimes_k \Omega_{K/k}^i).
    \]
    We note that $\{(\boldv \otimes d\boldx)^A\}_{A \in \Theta_{i,m}^{J}}$ is a basis of $\Sym_K^m (V \otimes_k \Omega_{K/k}^i)$.
    By using these notations, we define
    \[
        \Sym_{\mathcal{C}}^m (V \otimes_k \Omega^i_{R/k}(\log D)) \coloneqq \bigoplus_{A \in \Theta_{i,m}^{J}} R \cdot \frac{(\boldv \otimes d\boldsymbol{x})^A}{\prod_{s=1}^r x_s^{\lfloor c_s A(s) \rfloor}},
    \]
    where we write $A(s) \coloneqq \sum_{j \in J, \bolda \in \Sigma_i(s)} A(j,\bolda)$.
    As in the proof of Definition-Lemma \ref{deflem:differential module local}, this $R$-submodule is independent of the choice of $x_1,\dots, x_N$ and $\{v_{j}\}_{j \in J}$.
\end{rem}

\begin{defn}
    Let $(X,D)$ be a pair over a field $k$ with geometrically snc support.
    We define a locally free $\sO_X$-module $\Sym^m_{\mathcal{C}} \Omega^i_{X/k}(\log D)$ as a subsheaf of the constant sheaf $\Sym^m_K \Omega^i_{K(X)/k}$ on $X$ satisfying
    \[
    (\Sym^m_{\mathcal{C}} \Omega^i_{X/k}(\log D))_x = \Sym^m_{\mathcal{C}} \Omega^i_{\sO_{X,x}/k}(\log D_x)
    \]
    at every point $x \in X$.
    The existence of such a subsheaf is ensured by Definition-Lemma \ref{deflem:differential module local}.
    Similarly, for a finite dimensional $k$-module $V$, we define the locally free $\sO_X$-module $\Sym^m_{\mathcal{C}} (V \otimes_k \Omega^i_{X/k}(\log D))$ by 
        \[
    (\Sym^m_{\mathcal{C}} (V \otimes_k \Omega^i_{X/k}(\log D)))_x = \Sym^m_{\mathcal{C}} (V \otimes_k  \Omega^i_{\sO_{X,x}/k}(\log D_x))
    \]

\end{defn}

\begin{rem}\label{rem:C-differential}
    For a pair $(X,D)$ over a field $k$ with geometrically snc support, the following hold.
    \begin{enumerate}[label=\textup{(\arabic*)}]
    \item If all the coefficients of $D$ is contained in the set \[\{1-1/n \mid n \in \mathbb{Z}_{>0}\quad\text{with}\quad p\nmid n\} \cup \{1\},\] then the definition of $\Sym_{\mathcal{C}}^m \Omega_{X/k}^i (\log D)$ coincides with \cite[Definition 3.5]{JK}.
    \item If all the coefficients of $D$ are one, then we have
    \[
    \Sym_{\mathcal{C}}^m \Omega_{X/k}^i (\log D) = \Sym^m_{\sO_X} \Omega_{X/k}^i(\log D).
    \]
    \item If $\dim X=1$ and $X$ is proper over $k$, then we have
    \[
    \Sym^m_{\mathcal{C}} (V \otimes_k \Omega^i_{X/k}(\log D)) \cong \left\{\begin{array}{lll} 
    \Sym^m_k V \otimes_k \sO_X & \textup{($i=0$ or $m=0$)}\\
    \Sym^m_k V \otimes_k \sO_X(\lfloor m(K_X+D) \rfloor) & \textup{($i=1$ and $m \neq 0$)} \\ 
    0 & \textup{($i\geq 2$ and $m \neq 0$)} \end{array} \right., 
    \]
    where $K_X$ denotes a canonical divisor.
    \end{enumerate}
\end{rem}

\begin{defn}
    Let $(X,D)$ be a pair over a field $k$.
    We further assume that $X$ is geometrically regular over $k$ in codimension one and each component of $D$ is geometrically reduced over $k$.
    Then for every integers $i,m>0$, we define
    \[
    \Sym^{[m]}_{\mathcal{C}} \Omega^{[i]}_{X/k}(\log D) \coloneqq \iota_* \Sym^m_{\mathcal{C}} \Omega^i_{U/k}(\log D|_U),
    \]
    where $U \subseteq X$ is the maximal locus where $(U,D|_U)$ has geometrically snc support over $k$ and $\iota \colon U \into X$ is the inclusion.
    Note that the complement of $U$ has codimension at least two, and thus $\Sym^{[m]}_{\mathcal{C}} \Omega^{[i]}_{X/k}(\log D)$ is a reflexive $\sO_X$-module.
\end{defn}


\subsection{Basic properties for \texorpdfstring{$\mathcal{C}$}{C}-differentials}
We first recall the residue sequence \eqref{eq:res seq} and the restriction sequence \eqref{eq:restr seq}.

Let $(X,D=\sum_{i=0}^r a_i D_i)$ be a pair over a field $k$ with geometrically snc support
and $\iota\colon  D_0 \into X$ be the inclusion map.
Assume that $a_0=1$, and set $\widetilde{D} \coloneqq D-D_0 = \sum_{i=1}^n a_i D_i$.

As in \cite[Properties 2.3 (b)]{EV}, by using local coordinates, we can construct the residue map 
\[
\mathrm{res}^1_{D_0} : \Omega^i_{X/k}(\log \lfloor D \rfloor) \to \iota_* \Omega^{i-1}_{D_0/k}(\log \lfloor \widetilde{D}|_{D_0} \rfloor)
\]
which fits into the residue exact sequence
\begin{align}\label{eq:res seq}
0 \to \Omega^i_{X/k}(\log \lfloor \widetilde{D} \rfloor) \to \Omega^i_{X/k}(\log \lfloor D \rfloor) \xrightarrow{\res^{1}_{D_0}} \iota_* \Omega^{i-1}_{D_0/k}(\log \lfloor \widetilde{D}|_{D_0} \rfloor) \to 0.
\end{align}
Using local coordinates again, we can also define the residue map
\begin{align}\label{map:res}
\res^{m}_{D_0} \colon \Sym^m_{\mathcal{C}}\Omega^i_{X/k}(\log \widetilde{D})\onto \iota_* \Sym^m_{\mathcal{C}} \Omega^{i-1}_{D_0/k}(\log \widetilde{D}|_{D_0}),
\end{align}
which coincides with $\Sym^m(\mathrm{res}_{D_0}^1)$ on $X \setminus \Supp(\widetilde{D})$.

Moreover, the natural surjection $\sO_X \onto \iota_* \sO_{D_0}$ induces the restriction map
\begin{align}\label{map:restr}
\restr^{m}_{D_0} \colon \Sym^m_{\mathcal{C}}\Omega^i_{X/k}(\log D) \onto \iota_* \Sym^m_{\mathcal{C}} \Omega^{i}_{D_0/k}(\log \widetilde{D}|_{D_0}).
\end{align}
When $m=1$, it fits int the restriction exact sequence (cf.~\cite[Properties 2.3 (c)]{EV}).
\begin{align}\label{eq:restr seq}
0 \to \Omega^i_{X/k}(\log \lfloor D \rfloor)(-D_0) \to \Omega^i_{X/k}(\log \lfloor \widetilde{D} \rfloor) \xrightarrow{\restr^{1}_{D_0}} \iota_* \Omega^{i}_{D_0/k}(\log \lfloor \widetilde{D}|_{D_0} \rfloor) \to 0.
\end{align}

Next, we consider the behavior of $\mathcal{C}$-differentials under the change of base fields.

\begin{lem}\label{lem:change of field pre}
    Let $\kappa/k$ and $K/\kappa$ be a finitely generated separable field extensions.
    Then for every integers $i,m,p,l \ge 0$, there are $K$-submodules
        \[
        W^{p,l}=W_{i,m}^{p,l} \subseteq \Sym_{K}^m\Omega_{K/k}^i
        \]
        with the following properties:
        \begin{enumerate}[label=\textup{(\roman*)}]
            \item 
            We have $W^{p,l}_{i,m}/W^{p,l+1}_{i,m} \cong W^{p+1,0}_{i,l} \otimes_K \Sym^{m-l}_{K} (\Omega^{p}_{\kappa/k} \otimes_{\kappa} \Omega^{i-p}_{K/\kappa})$.
            \item One has 
            \[
            W^{p,l}_{i,m} = \left\{ \begin{array}{ll} 0 & \textup{(if $ l \ge m+1$, or $ p \ge i+1$ and $m \neq 0$)} \\
            K & \textup{(if $ l =m= 0$)} \\
            \Sym^m_K \Omega^i_{K/k} & \textup{(if $p=l=0$)} 
            \end{array}\right.
            \]
        \end{enumerate}
\end{lem}

\begin{proof}
Noting that 
\[
0 \to \Omega_{\kappa/k}^1 \otimes_{\kappa} K \to \Omega_{K/k}^1 \to \Omega_{K/\kappa}^1 \to 0
\]
is exact (\stacksproj{02K4}), the image $V^p=V^p_{i} \coloneqq \Im(\Phi^p)$ of the natural map $\Phi^p \colon \Omega_{\kappa/k}^p  \otimes_{\kappa} \Omega_{K/k}^{i-p} \to \Omega^{i}_{K/k}$ defined by 
\[
\Phi^p((df_1\wedge \dots df_p) \otimes (dg_1 \wedge \dots dg_{i-p}) ) = df_1\wedge \dots df_p \wedge dg_1 \wedge \dots dg_{i-p}
\]
gives a filtration
\[
\Omega_{K/k}^i =V^0 \supseteq V^1 \supseteq \dots \supseteq V^{i+1} =0
\]
with $V^p/V^{p+1} \cong \Omega_{\kappa/k}^p \otimes_{\kappa} \Omega_{K/\kappa}^{i-p}$ (c.f.~\cite[I. Exercise 5.16 (d)]{Har}).

Next, since we have the exact sequence
\[
0 \to V^{p+1} \to V^p \to \Omega_{\kappa/k}^p \otimes_{\kappa} \Omega_{K/\kappa}^{i-p} \to 0,
\]
the image $W^{p,l}=W^{p,l}_{i,m} \coloneqq \Im(\Psi^{p,l}_{m,i})$ of the natural map $\Psi^{p,l}_{i,m} \colon \Sym^l_K V^{p+1} \otimes_{K} \Sym^{m-l}_K V^p \to \Sym^m_K V^p$ defined by 
\[
\Psi^{p,l}_{i,m}(\alpha \otimes \beta ) = \alpha \beta
\]
gives a filtration
\[
\Sym^m_K V^p =W^{p,0} \supseteq W^{p,1} \supseteq \dots \supseteq W^{p,m+1} =0
\]
with 
\begin{align*}
W^{p,l}/W^{p,l+1} & \cong \Sym^l_K V^{p+1} \otimes_{K} \Sym^{m-l}_K (\Omega_{\kappa/k}^p \otimes_{\kappa} \Omega^{i-p}_{K/\kappa}) \\
& = W^{p+1,0}_{i,l} \otimes_K \Sym^{m-l}_K (\Omega_{\kappa/k}^p \otimes_{\kappa} \Omega^{i-p}_{K/\kappa}).
\end{align*}
(c.f.~\cite[I. Exercise 5.16 (c)]{Har}).

\end{proof}

\begin{rem}\label{rem:change of field}
We give an explicit description of the isomorphism 
\[
f=f^{p,l}_{i,m} \colon W^{p+1,0}_{i,l} \otimes_{K} \Sym^{m-l}_K (\Omega_{\kappa/k}^p \otimes_{\kappa} \Omega^{i-p}_{K/\kappa}) \xrightarrow{\sim} W^{p,l}_{i,m}/W^{p,l+1}_{i,m}
\]
in Lemma \ref{lem:change of field pre} (i).

\textbf{Step 1}:
We first give an explicit description of the isomorphism $ V^p/V^{p+1} \cong \Omega_{\kappa/k}^p \otimes_{\kappa} \Omega_{K/\kappa}^{i-p}$.
Let $dx_1, \dots, dx_{M}$ be a basis of $\Omega_{\kappa/k}^1$ and $dx_{M+1}, \dots, dx_{N}$ be a basis of $\Omega_{K/\kappa}^1$.
Then we have
\[
V^p = \bigoplus_{\bolda \in \Sigma_{i,p}^{N,M}} K d\boldx_{\bolda},
\]
where $\Sigma_{i,p}^{N,M}$ is the subset of $\Sigma_i=\Sigma^N_i$ defined by 
\begin{align*}
\Sigma_{i,p}^{N,M} & = \{\bolda =(a_1,\dots,a_i) \in \Sigma_i^N \mid a_p \le M \}.
\end{align*}
See Notation \ref{Notation:basis of Omega} for the definition of $\Sigma_{i}^N$.

The isomorphism $ V^p/V^{p+1} \cong \Omega_{\kappa/k}^p \otimes_{\kappa} \Omega_{K/\kappa}^{i-p}$ is induced by the morphism
\[
V^p \to \Omega_{\kappa/k}^p \otimes_{\kappa} \Omega_{K/\kappa}^{i-p}\ ; \ d\boldx_{\bolda} \mapsto (dx_{a_1} \wedge \dots \wedge dx_{a_{p}}) \otimes (dx_{a_{p+1}} \wedge \dots \wedge dx_{a_i}).
\]

\textbf{Step 2}: In this step, we provide a basis of $\Sym^{m-l}_{K}(\Omega^p_{\kappa/k} \otimes_{\kappa} \Omega_{K/\kappa}^{i-p})$.
We note that $\Omega^p_{\kappa/k} = \bigoplus_{\bolda} \kappa d\boldx_{\bolda}$, where $\bolda$ runs through all elements of 
\[
\choosing{[M]}{p} \coloneqq \{\bolda=(a_1, \dots, a_p) \mid 1 \le a_1< \dots < a_p \le M \}
\]
and $\Omega^{i-p}_{K/\kappa} = \bigoplus_{\boldb} d\boldx_{\boldb}$, where $\boldb$ runs through all elements of 
\[
\choosing{[N] \setminus[M]}{i-p} \coloneqq \{\boldb=(b_1, \dots, b_{i-p}) \mid M+1 \le b_1< \dots < b_{i-p} \le N \}.
\]
Therefore, the $K$-vector space $\Sym^{m-l}_{K}(\Omega^p_{\kappa/k} \otimes_{\kappa} \Omega_{K/\kappa}^{i-p})$ is spanned by $\{(d\boldx \otimes d\boldx)^C\}_{C}$,
where $C$ runs through all elements in
\begin{align*}
\Delta_{i,m}^{p,l} & \coloneqq \Map( \choosing{[M]}{p} \times \choosing{[N] \setminus [M]}{i-p}, \Z_{\ge 0})_{m-l} \\
& = \left\{ C \colon \choosing{[M]}{p} \times \choosing{[N] \setminus [M]}{i-p} \to \Z_{\ge 0} \middle| \sum_{(\bolda,\boldb)} C(\bolda,\boldb) =m-l \right\}.
\end{align*}
and we write 
\[
(d\boldx \otimes d\boldx)^C = \prod_{(\bolda,\boldb)} (d\boldx_{\bolda} \otimes d\boldx_{\boldb})^{C(\bolda,\boldb)} \in \Sym^{m-l}_{K}(\Omega^p_{\kappa/k} \otimes_{\kappa} \Omega_{K/\kappa}^{i-p}),
\]
where $(\bolda, \boldb)$ runs through all elements of $\choosing{[M]}{p} \times \choosing{[N] \setminus [M]}{i-p}$.

\textbf{Step 3}: In this step, we give an explicit description of the isomorphism $f$.
By the definition of $W^{p,l}_{i,m}$, we have
\[
W^{p,l}_{i,m} = \bigoplus_{A \in \Lambda_{i,m}^{p,l}} K d\boldx^A,
\]
where $\Lambda^{p,l}_{i,m}$ is the subset of $\Theta_{i,m} =\Map (\Sigma_i, \Z_{\ge 0})_m$ consists of all elements $A \in \Theta_{i,m}$ such that 
\begin{itemize}
    \item every element $\bolda \in \Sigma_{i}$ with $A (\bolda) \neq 0$ is contained in $\Sigma_{i,p}^{N,M}$, and
    \item we have 
    \[
    \sum_{\bolda \in \Sigma_{i, p+1}^{N,M}} A (\bolda)  \ge l.
    \]
\end{itemize}

On the other hand, by identifying $\choosing{[M]}{p} \times \choosing{[N] \setminus [M]}{i-p}$ as the subset $\Sigma^{N,M}_{i,p} \setminus \Sigma^{N,M}_{i,p+1}$ of $\Sigma_{i}$, we may consider every element $C \in \Delta^{p,l}_{i,m}$ as a map from $\Sigma_{i}$ to $\Z_{\ge 0}$ whose restriction to the complement of $\Sigma^{N,M}_{i,p} \setminus \Sigma^{N,M}_{i,p+1}$ is zero.

Then $f$ fits into the following commutative diagram
\[
\xymatrix{
W^{p+1,0}_{i,l} \otimes_{K} \Sym^{m-l}_K (\Omega_{\kappa/k}^p \otimes_{\kappa} \Omega^{i-p}_{K/\kappa}) \ar@{}[d]|*=0[@]{\cong} \ar^-{f}[r] &  W^{p,l}_{i,m}/W^{p,l+1}_{i,m} \ar@{}[d]|*=0[@]{\cong}\\
\displaystyle \bigoplus_{B \in \Lambda^{p+1,0}_{i,l}}   K d\boldx^B \otimes_K \bigoplus_{C \in \Delta^{p,l}_{i,m}} K (d\boldx \otimes d\boldx)^C \ar^-{g}[r]& \displaystyle \bigoplus_{A \in \Lambda^{p,l}_{i,m} \setminus \Lambda^{p,l+1}_{i,m}} K \overline{d\boldx ^A} 
},\]
where $g$ maps $d\boldx^B \otimes (d \boldx \otimes d\boldx)^C$ to $\overline{d\boldx^{B+C}}$.
\end{rem}

\begin{lem}\label{lem:change of fields}
    Let $k$ be field, $\kappa/k$ be a finitely generated separable field extension.
    Let $(X,D=\sum_i D_i)$ be a pair over $\kappa$ with geometrically snc support.
    Then $(X,D)$ is also a pair over $k$ with geometrically snc support, and for every $p,l \ge 0$, there is locally free $\sO_X$-subsmodule
        \[
        F^{p,l}=F_{i,m}^{p,l} \subseteq \Sym_{\mathcal{C}}^m\Omega_{X/k}^i(\log D)
        \]
        with the following properties:
        \begin{enumerate}[label=\textup{(\roman*)}]
            \item 
            There is an injection 
            \[
            f^{p,l}_{i,m} \colon F^{p,l}_{i,m}/F^{p,l+1}_{i,m} \into F^{p+1,0}_{i,l}(\lceil D \rceil) \otimes_K \Sym^{m-l}_{\mathcal{C}} (\Omega^{p}_{\kappa/k} \otimes_{\kappa} \Omega^{i-p}_{K/\kappa}(\log D))
            \]
            whose image contains $F^{p+1,0}_{i,l} \otimes_K \Sym^{m-l}_{\mathcal{C}} (\Omega^{p}_{\kappa/k} \otimes_{\kappa} \Omega^{i-p}_{K/\kappa}(\log D))$.
            (See Remark \ref{rem:variant of C-diff} for the definition of $\Sym^{m-l}_{\mathcal{C}} (\Omega^{p}_{\kappa/k} \otimes_{\kappa} \Omega^{i-p}_{X/\kappa}(\log D))$.)
            \item 
            If $ p \ge i-1$ or $m=l$, then $f^{p,l}_{i,m}$ induces 
            \[
            F^{p,l}_{i,m}/F^{p,l+1}_{i,m} \cong F^{p+1,0}_{i,l} \otimes_K \Sym^{m-l}_{\mathcal{C}} (\Omega^{p}_{\kappa/k} \otimes_{\kappa} \Omega^{i-p}_{K/\kappa}(\log D)).
            \]
            \item We have 
            \[
            F^{p,l}_{i,m} = \left\{ \begin{array}{ll} 0 & \textup{(if $ l \ge m+1$, or $ p \ge i+1$ and $m \neq 0$)} \\
            \sO_X & \textup{(if $l=m=0$)} \\
            \Sym^m_{\mathcal{C}} \Omega^i_{X/k}(\log D) & \textup{(if $p=l=0$)} 
            \end{array}\right.
            \]
        \end{enumerate}
\end{lem}

\begin{proof}
Since $\kappa$ is geometrically regular over $k$ \stacksproj{07EQ}, the pair $(X,D)$ has geometrically snc support over $k$ \stacksproj{07QI}.

Let $K=K(X)$ be the function field of $X$ and $\underline{W^{p,l}_{i,m}}_{X}$ be the constant sheaf associated to the sub $K$-module $W^{p,l}_{i,m} \subseteq \Sym^m_{K} \Omega^i_{K/k}$ defined in Lemma \ref{lem:change of field pre}.
We define
\[
F^{p,l}_{i,m} \coloneqq \underline{W^{p,l}_{i,m}}_{X} \cap \Sym_{\mathcal{C}}^m\Omega_{X/k}^i(\log D).
\]

In order to prove that $F^{p,l}_{i,m}$ is locally free, we provide a local description of $F^{p,l}_{i,m}$.
Suppose that $X=\Spec R$ for a local ring $R$.
Take elements $x_{1}, \dots, x_{M} \in \kappa$ so that we have
\[
\Omega^1_{\kappa/k} = \bigoplus_{i=1}^M \kappa dx_{i}.
\]
We also take $x_{M+1}, \dots, x_{M+r}, x_{M+r+1} , \dots, x_N \in R$ such that $x_{M+i}$ is an equation of $D_i$ for $i=1,2, \dots r$ and 
\[
\Omega^1_{R/\kappa} = \bigoplus_{i=M+1}^N R dx_i.
\]
Since $R$ is essentially of finite type over $\kappa$, we may take a smooth variety $V$ over $k$ and a point $P \in V$ such that $R \cong \sO_{V,P}$.
By localizing the first exact sequence \cite[\href{https://stacks.math.columbia.edu/tag/02K4}{Tag 02K4}]{stacks-project}, we have the equation
\begin{align}\label{eq:change of field local coodinate}
\Omega^1_{R/k} = \bigoplus_{i=1}^{N} R dx_i.
\end{align}
Then we have
\begin{align*}
    F^{p,l}_{i,m} & =W^{p,l}_{i,m} \cap \Sym^m_{\mathcal{C}} \Omega^i_{X/k}(\log D) \\
    & = \bigoplus_{A \in \Lambda^{p,l}_{i,m}} K d\boldx^A \cap \bigoplus_{A \in \Theta_{i,m}}R \cdot \frac{d\boldx^A}{\prod_{s=M+1}^{M+r} x_s^{\lfloor c_s A(s) \rfloor}} \\
    & = \bigoplus_{A \in \Lambda^{p,l}_{i,m}} R \cdot \frac{d\boldx^A}{\prod_{s=M+1}^{M+r} x_s^{\lfloor c_s A(s) \rfloor}},
\end{align*}
which is a free $R$-module.

We next verify the conditions (i), (ii) and (iii) in the assertion.
The condition (iii) follow from Lemma \ref{lem:change of field pre} (ii).
For (i), we first note that the natural morphism 
\[
F^{p,l}_{i,m}/F^{p.l+1}_{i,m} \to W^{p,l}_{i,m}/W^{p,l+1}_{i,m}
\]
is injective since we have
\begin{align*}
    F^{p,l+1}_{i,m} & =W^{p,l+1}_{i,m} \cap \Sym^m_{\mathcal{C}} \Omega^i_{X/k}(\log D) \\
    & =W^{p,l+1}_{i,m} \cap (W^{p,l}_{i,m} \cap \Sym^m_{\mathcal{C}} \Omega^i_{X/k}(\log D)) \\
    & =W^{p,l+1}_{i,m} \cap F^{p,l}_{i,m}.
\end{align*}
Therefore, it suffices to show that $F^{p,l}_{i,m}/F^{p,l+1}_{i,m}$ satisfies the inclusions 
\begin{align}\label{eq:inclusion of filtraion}
\begin{split}
F^{p,l}_{i,m}/F^{p,l+1}_{i,m} & \subseteq f( F^{p+1,0}_{i,l}(\lceil D \rceil) \otimes \Sym^{m-l}_{\mathcal{C}} (\Omega^p_{\kappa/k} \otimes_{\kappa} \Omega^{i-p}_{X/\kappa}(\log D))) \textup{ and} \\
F^{p,l}_{i,m}/F^{p,l+1}_{i,m} & \supseteq f( F^{p+1,0}_{i,l} \otimes \Sym^{m-l}_{\mathcal{C}} (\Omega^p_{\kappa/k} \otimes_{\kappa} \Omega^{i-p}_{X/\kappa}(\log D))),
\end{split}
\end{align}
where $f=f^{p,l}_{i,m}$ is the isomorphism in Remark \ref{rem:change of field}.

After localization, we may adopt the notation in \eqref{eq:change of field local coodinate} and Remark \ref{rem:change of field}.
Then $F^{p,l}_{i,m}/F^{p,l+1}_{i,m}$ is a free $\sO_X$-module with basis
\[
\left\{\frac{1 }{\prod_{s=M+1}^{M+r} x_s^{\lfloor c_s A(s) \rfloor}} \overline{d\boldx^A} \middle| A \in \Lambda^{p,l}_{i,m} \setminus \Lambda^{p,l+1}_{i,m} \right\}
\]
On the other hand, $F^{p+1,0}_{i,l} \otimes \Sym^{m-l}_{\mathcal{C}} (\Omega^p_{\kappa/k} \otimes_{\kappa}  \Omega^{i-p}_{X/\kappa}(\log D))$ is a free $\sO_X$-module with basis
\[
\left\{ \frac{1 }{\prod_{s=M+1}^{M+r} x_s^{\lfloor c_s B(s) \rfloor + \lfloor c_sC(s) \rfloor}} (d \boldx^B \otimes (d \boldx \otimes d \boldx)^C) \middle| B \in \Lambda^{p+1,0}_{i,l}, C \in \Delta^{p,l}_{i,m} \right\}.
\]
Then the inclusions \eqref{eq:inclusion of filtraion} follow from the inequality
\[
\lfloor c_s B(s) \rfloor + \lfloor c_sC(s) \rfloor \le \lfloor c_s(B(s)+C(s)) \rfloor \le \lfloor c_s B(s) \rfloor + \lfloor c_sC(s) \rfloor+1
\]
for every $B,C$ and $s$.

We finally prove the assertion in (ii).
By the above argument, it suffices to show the equality
\[
\lfloor c_s B(s) \rfloor + \lfloor c_sC(s) \rfloor = \lfloor c_s(B(s)+C(s)) \rfloor 
\]
for every $s=M+1, \dots, M+r$.
If $p \ge i-1$, then for every element $B \in \Lambda^{p+1,0}_{i,l}$ and $s>M$, we have $B(s)=0$.
In this case, we have
\[
\lfloor c_s B(s) \rfloor + \lfloor c_sC(s) \rfloor = \lfloor c_sC(s) \rfloor = \lfloor c_s(B(s)+C(s)) \rfloor.
\]
If $m=l$, then this equality follows from the fact that $\Delta^{p,m}_{i,m}=\{0\}$.
\qedhere
\end{proof}

Let $G$ be a finite group and $X$ be an integral scheme over a field $k$. Assume that $G$ acts on $X$ via a group homomorphism $\phi\colon G \to \mathrm{Aut}_k(X)$. 
For each $g \in G $, the automorphism $\phi_g \colon X \to X $ induces a map $d\phi_g\colon \phi_g^* \Omega_{X/k}^1 \to \Omega_{X/k}^1 $.  
By taking the adjoint, we obtain a homomorphism  
\[
\theta_g\colon \Omega_{X/k}^1 \to (\phi_g)_* \Omega_{X/k}^1,  
\]  
which endows $\Omega_{X/k}^1 $ with the structure of a $G $-sheaf (see \cite[Definition B.1]{GKKP} for the definition).  
By localizing $\theta_g $, we obtain a $G$-sheaf structure on the constant sheaf of $\Omega_{K(X)/k}^1 $. 
Furthermore, for all $i, m \geq 0$, the constant sheaf associated with $\Sym^m_{K(X)}\Omega_{K(X)/k}^i$ can naturally be regarded as a $G$-sheaf.

We further assume that $X=\Spec R$ is affine.
Consider the quotient $Y\coloneqq\Spec R^G$ with the trivial $G$-action.
For a quasi-coherent $G$-sheaf $\mathcal{G}$ of $\sO_X$-modules, the push-forward $\pi_*\mathcal{G}$ is a quasi-coherent $G$-sheaf of $\sO_Y$-modules, and its $G$-invariant subsheaf is denoted by $(\pi_*\mathcal{G})^G$.
On the other hand, for a quasi-coherent sheaf $\mathcal{F}$ on $Y$ with the trivial $G$-action, the pullback $\pi^*\mathcal{F}$ is a quasi-coherent $G$-sheaf on $X$.
If $\mathcal{F}$ is a locally free coherent sheaf on $Y$, then we have
\begin{align}\label{eq:G-sheaf}
(\pi_*\pi^*\mathcal{F})^G = \mathcal{F}.
\end{align}

\begin{lem}\label{lem:G-sheaf}
    Let $G$ be a finite group acting on an integral scheme $X=\Spec R$ essentially of finite type over a perfect field $k$, and let $\pi \colon X \to Y = \Spec R^G$ be the quotient morphism.
    \begin{enumerate}[label=\textup{(\arabic*)}]
    \item If $L\coloneqq K(X)$ is separable over $K\coloneqq K(Y)$, then we have an equality
    \[
    (\Sym^m_{L}\Omega_{L/k}^i)^G = \Sym^m_{K}\Omega_{K/k}^i.
    \]
    as submodules of $\Sym^m_{L} \Omega^i_{L/k}$.
    \item If $\pi$ is \'{e}tale and $D$ is a $\Q$-divisor on $Y$ such that $(Y,D)$ is a pair over $k$ with geometrically snc support, then for integers $i,m \geq0$, we have an equality
    \[
    (\pi_* \Sym^m_{\mathcal{C}}\Omega_{X/k}^i(\log \pi^{-1}D))^G = \Sym^m_{\mathcal{C}}\Omega_{Y/k}^i(\log D).
    \]
    as subsheaves of the constant sheaf of $\Sym^m_{L} \Omega^i_{L/k}$.
    \end{enumerate}
\end{lem}

\begin{proof}
    We first prove (2).
    Since $\pi$ is \'etale, we have $\Omega^1_{X/Y}=0$ \cite[\href{https://stacks.math.columbia.edu/tag/02GU}{Tag 02GU}]{stacks-project}.
    Combining this with the first exact sequence \cite[\href{https://stacks.math.columbia.edu/tag/02K4}{Tag 02K4}]{stacks-project}, we have $\Omega^1_{X/k} \cong \pi^* \Omega^1_{Y/k}$, and in particular, one has $\Omega_{L/k}^1 \cong \Omega_{K/k}^1 \otimes_K L$.
    Therefore, all sheaves in the assertion are subsheaves of the constant sheaf associated to $\Sym_{L}^m\Omega_{L/k}^i \cong \Sym_K^m\Omega_{K/k}^i \otimes_K L$.

    In order to prove the equality in (2), after shrinking $Y$, we may assume that there exists a sequence $x_1, \dots, x_N \in H^0(Y,\sO_Y)$ such that
    \[
    \Omega^1_{Y/k} = \bigoplus_{i=1}^N \sO_Y dx_i\quad\textup{and}\quad D = \Div(x_1 \cdots x_r).
    \]
    Then we have the natural isomorphism
    \[
    \Sym_{\mathcal{C}}^m\Omega_{X/k}^i(\log \pi^{-1}D) \cong \pi^*\Sym_{\mathcal{C}}^m\Omega_{Y/k}^i(\log D)
    \]
    as $G$-sheaves.
    The assertion now follows from the equation \eqref{eq:G-sheaf}.

    For (1), after shrinking $Y$, we may assume that $Y$ is geometrically regular over $k$ and $\pi$ is \'etale.
    Then the assertion in (1) follows from (2).
\end{proof}

\section{Logarithmic extension theorem in dimension two}\label{Section:LET for surfaces}

In this section, we prove Theorem \ref{Introthm:2-dim} based on fundamental results in Section \ref{Section:Logarithmic differential sheaf}.

\subsection{Residue maps and restriction maps}
In this subsection, we generalize the residue sequence \eqref{eq:res seq} and the restriction sequence \eqref{eq:restr seq} to the case where the pair is two-dimensional and tamely dlt (Theorem \ref{thm:reside and restr ex seq}).
This provides a generalization of Graf's result \cite[Theorems 1.4 and 1.5]{Gra} to the case of imperfect residue fields.

\begin{lem}\label{lem:key inclusion in DVR case}
    Let $(R,\m=(x))$ and $(S,\n=(y))$ be discrete valuation rings essentially of finite type over a perfect field $k$ and $\phi\colon R \into S$ be a finite extension of degree prime to $p=\chara(k)$.
    Suppose that $K$ (resp.~$L$) is the fraction field of $R$ (resp.~$S$) and $D$ be the $\Q$-divisor $(1-\frac{1}{e})\{\m\}$ on $\Spec R$, where $e$ is the ramification index of $S$ over $R$.
    Then for integers $i,m \geq0$, we have an equality
        \[
        \Sym^m_S \Omega^i_{S/k} \cap \Sym^m_K \Omega^i_{K/k} = \Sym^m_{\mathcal{C}} \Omega^i_{R/k}(\log D)
        \]
        as submodules of $\Sym^m_L \Omega^i_{L/k}$.
\end{lem}

\begin{proof}
Since $L$ is separable over $K$, as in the proof of Lemma \ref{lem:G-sheaf}, we have the natural isomorphism 
    \[
    \Sym_L^m \Omega_{L/k}^i \cong (\Sym^m_K \Omega_{K/k}^i) \otimes_K L .
    \]
    This shows that all the modules appearing in the assertion are considered as submodules of $\Sym^m_L \Omega^i_{L/k}$.
    We also note that the ramification index $e$ is not divisible by $p$ since one has the equation $[L:K] = e [S/\n : R/\m] $ (cf.~\cite[\href{https://stacks.math.columbia.edu/tag/09E8}{Tag 09E8}]{stacks-project}).
    
    \textbf{Step 1}: 
    In this step, we will find bases of $\Omega^1_{R/k}$ and $\Omega^1_{S/k}$.

    By \cite[Theorem 28.3]{Matsumura} and \cite[\href{https://stacks.math.columbia.edu/tag/090W}{Tag 090W}]{stacks-project}, we can find a field $\ell\subseteq R$ containing $k$ such that $R/\m$ is finite and separable over $\ell$.
    The second exact sequence for $\ell \into R \onto R/\m$
    shows that $\Omega^1_{R/\ell}$ is generated by a single element $dx$.
    On the other hand, $\Omega^1_{K/\ell} \neq 0$ as $K$ is transcendental over $\ell$.
    This implies that $dx \in \Omega^1_{R/\ell}$ is a torsion-free element and therefore, we have 
    \[
    \Omega^1_{R/\ell} = R dx \cong R.
    \]
    Take elements $x_2, \dots, x_{r} \in \ell$ so that 
    \[
    \Omega^1_{\ell/k} = \bigoplus_{i=2}^r \ell dx_i.
    \]

    Since $R/(x)$ is geometrically regular over $\ell$, so is $R$.
    It then follows from \stacksproj{02K4} that the sequence $x_1 \coloneqq x, x_2, \dots, x_r \in R$ gives a basis $dx_1, \dots, dx_r$ of $\Omega_{R/k}^1$ over $R$.
    Similarly, the sequence 
    \[
    y_1 \coloneqq y, y_2\coloneqq x_2, y_3 \coloneqq x_3 , \dots, y_r\coloneqq x_r \in S
    \]
    provides a basis of $\Omega^1_{S/k}$.

    \textbf{Step 2}:
    We use the notation as in Definition-Lemma \ref{deflem:differential module local}.
    In particular, $\Sym^m_{\mathcal{C}} \Omega^i_{R/k}(\log D)$ is a free $R$-module with a basis $\{d\boldx^A / x_1^{\lfloor (1-1/e)A(1) \rfloor}\}_{A \in \Theta_{i,m}}$ and $\Sym^m_S \Omega^i_{S/k}$ is a free $S$-module with a basis $\{d\boldsymbol{y}^B\}_{B \in \Theta_{i,m}}$.

    By the definition of the ramification index, there exists a unit element $\epsilon \in S$ such that $x_1 = \epsilon y_1^e$.
    Take elements $s_1, \dots ,s_r \in S$ such that we have $d\epsilon = \sum_i s_i dy_i$.
    We then have the equation
    \begin{align*}
    d x_1 &= (e \epsilon y_1^{e-1}+ s_1y_1^e )dy_1 + \sum_{i=2}^r s_i y_1^e dy_i \\
    & = \epsilon' y_1^{e-1} dy_1 + \sum_{i=2}^r s_i y_1^edy_i \\
    & = y_1^{e-1} \cdot ( \epsilon' dy_1 + \sum_{i=2}^r s_iy_1 dy_i),
    \end{align*}
    where $\epsilon' \coloneqq e \epsilon + s_1 y_1 \in S$ is a unit element.
    By using this equation, for every $A \in \Theta_{i,m}$, we can write
    \begin{align}\label{eq:transformation rule}
    d\boldx^A = y_1^{(e-1)A(1)} \cdot ( (\epsilon' )^{A(1)}d\boldsymbol{y}^A + \sum_{B \in \Theta_{i,m}, B(1) <A(1)} \beta_A(B)d\boldsymbol{y}^B)
    \end{align}
    for some $\beta_A(B) \in S$.

    \textbf{Step 3}:
    In this step, we show the containment
    \[
        \Sym^m_S \Omega^i_{S/k} \cap \Sym^m_K \Omega^i_{K/k} \supseteq \Sym^m_{\mathcal{C}} \Omega^i_{R/k}(\log D).
    \]
    Take an element $\omega \in \Sym^m_{\mathcal{C}} \Omega^i_{R/k}(\log D)$ and we write 
    \[
    \omega = \sum_{A \in \Theta_{i,m}} \alpha(A) d\boldx^A
    \]
    with $\alpha(A) \in x_1^{-\lfloor (1-1/e) A(1) \rfloor}R$.
    It then follows from \eqref{eq:transformation rule} that we have
    \begin{align*}
    \alpha(A) d \boldx^A &\in x_1^{-\lfloor (1-1/e) A(1) \rfloor} y_1^{(e-1)A(1)} \Sym^m_S\Omega^i_{S/k} \\
    & = y_1^{-e\lfloor (1-1/e) A(1) \rfloor + (e-1)A(1)}\Sym^m_S\Omega^i_{S/k} \\
    & \subseteq  y_1^{-e(1-1/e) A(1) + (e-1)A(1)}\Sym^m_S\Omega^i_{S/k} \\
    & = \Sym^m_S\Omega^i_{S/k}.
    \end{align*}
    Therefore, we have $\omega \in \Sym^m_S\Omega_{S/k}^i$ as desired.

    \textbf{Step 4}:
    In this step, we verify the converse inclusion.
    Take an element $\omega \in \Sym^m_S \Omega^i_{S/k} \cap \Sym^m_K \Omega^i_{K/k}$.    
    Since $\omega \in \Sym^m_K \Omega^i_{K/k}$, we may write 
    \[
    \omega = \sum_{A \in \Theta_{i,m}} \alpha(A) d\boldx^A
    \]
    with $\alpha(A) \in K$.
    It suffices to prove the containment 
    \[
    \alpha(A) \in \frac{1}{x_1^{\lfloor (1-1/e) A(1) \rfloor}} R
    \]
    by descending induction on $A(1)$.

    We fix $A \in \Theta_{i,m}$ and assume that the containment holds for all $B$ with $B(1) > A(1)$.
    In particular, for such $B$, we have
    \[
    y_1^{(e-1)A(1)} \alpha(B) \in y_1^{e \lfloor (1-1/e)B(1) \rfloor} \alpha(B) S =x_1^{\lfloor (1-1/e)B(1) \rfloor}\alpha(B)S \subseteq S. 
    \]
    On the other hand, it follows from the equation \eqref{eq:transformation rule} that the coefficient of $\omega$ at $d\boldsymbol{y}^A$ is 
    \begin{align}\label{eq:apply transformation rule}
    (\epsilon'y_1^{e-1})^{A(1)}\alpha(A) + \sum_{B \in \Theta_{i,m}, B(1)>A(1)} \beta_{A}(B) y_1^{(e-1)A(1)}\alpha(B).
    \end{align}
    Combining induction hypothesis with the assumption that $\omega \in \Sym^m_S \Omega^i_{S/k}$, we conclude that  
    \[
    y_1^{(e-1)A(1)}\alpha(A) \in S.
    \]
    Let $v \in \Z$ be the valuation of $\alpha(A)$ with respect to $R$.
    It then follows from the above containment that we have
    \[
    (e-1)A(1) + ev \geq0.
    \]
    Therefore, one has
    \[
    v \geq-\lfloor (1 - \frac{1}{e}) A(1) \rfloor,
    \]
    which  implies the desired containment.
\end{proof}

\begin{lem}\label{lem:quasi-etale cover}
    Let $(R,\m)$ be a two-dimensional normal local domain essentially of finite type over a perfect field $k$ of characteristic $p>0$.
    Let $P$ be a prime divisor on $X=\Spec R$ such that
    $(X,P)$ is plt and the Cartier index $\ell$ of $K_X+P$ is prime to $p$.
    We further assume that an $\ell$-th root of unity $\zeta=\sqrt[\ell]{1}$ is contained in $k$.
    Then the following assertions hold:
    \begin{enumerate}[label=\textup{(\arabic*)}]
    \item There exists a finite quasi-\'etale cyclic Galois cover 
    \[
    \gamma \colon Y \to X
    \]
    degree prime to $p$ such that both $Y$ and $Q \coloneqq (\gamma^{-1}(P))_{\mathrm{red}}$ are regular.
    \item    
    Let $G=\Z/\ell\Z$ be the Galois group.
    For every integers $i,m \geq0$ and every rational number $0 \leq c \leq 1$, we have
    \[
    (\gamma_* \Sym^m_{\mathcal{C}}\Omega^i_{Y/k}(\log cQ))^G = \Sym^{[m]}_{\mathcal{C}}\Omega^{[i]}_{X/k}(\log cP).
    \]
    \item For every integers $i,m \geq0$, the $\sO_Q$-module $\Sym^m_{\sO_{Q}}\Omega^i_{Q/k}$ is a $G$-sheaf and we have
    \[
    (\gamma_* \Sym^m_{\sO_{Q}} \Omega^i_{Q/k})^G = \Sym^m_{\mathcal{C}} \Omega^i_{P/k}(\log \mathrm{Diff}_P(0)). 
    \]
    where $\mathrm{Diff}_P(0)$ is the different of $0$ on $P$ \textup{\cite[Subsection 4.1]{Kol}} .
    \end{enumerate}
\end{lem}

\begin{proof}
    As for (1), we can take $Y$ as an index one cover, i.e.,  
    \[
    Y\coloneqq \Spec \big(\bigoplus_{n \geq0 } H^0(X, \sO_X(n(K_X+P)))T^n/(T^{\ell}/r-1)\big),
    \]
    where $\ell$ is the Cartier index of $K_X+P$ and $r$ is an element such that $\ell(K_X+P)=\Div(r)$.
    Since the function field $K(X)$ contains an root of unity $\zeta$, we have $K(Y)=K(X)(\sqrt[\ell]{r})$ is a Galois extension with Galois group $G=\Z/\ell \Z$.
    Noting that $(Y,Q)$ is plt \cite[Corollary 2.43 (2)]{Kol13}, it follows from \cite[Section 3.35]{Kol13} that $Q$ is regular.
    Since every discrepancy of $K_Y+Q=\Div(T)$ is an integer, 
    $(Y,Q)$ is canonical, and hence $Y$ is regular \cite[Theorem 2.29 (2)]{Kol13}. 
    Thus, we conclude (1).
    
    For (2), we first note that $\gamma$ is the quotient of $Y$ by the natural action of $G$.
    Since $Q$ is a $G$-invariant subscheme of $Y$, the sheaf $\Sym^m_{\mathcal{C}}\Omega^i_{Y/k}(\log cD)$ is a $G$-invariant subsheaf of the constant sheaf $\Sym^m_{K(Y)}\Omega^i_{K(Y)/k}$.
    Since both sheaves in the assertion are reflexive \cite[Lemma B.4]{GKKP}, it is enough to verify the equality in codimension one.
    Then the assertion follows from Lemma \ref{lem:G-sheaf} (2).
    
    We finally prove (3).
    Since $Q$ is a $G$-invariant subscheme of $Y$ and $\gamma$ induces a morphism $\gamma: Q \to P$, we have a natural inclusion $\sO_P \into (\gamma_*\sO_Q)^G$.
    Noting that the order $|G|=\ell$ is not divisible by $p$, by using the Raynolds operator, the functor $(\gamma_* -)^G$ is exact.
    It then follows from the commutative diagram
    \[
\begin{tikzcd}
\sO_X \arrow[r, ->>] \arrow[d, equal] & \sO_P \arrow[d, hookrightarrow]\\
 (\gamma_* \sO_Y)^G \arrow[r, ->>] & (\gamma_*\sO_Q)^G 
\end{tikzcd}
    \]
    that we have the equality $\sO_P=(\gamma_* \sO_Q)^G$.
    Applying Lemma \ref{lem:G-sheaf} (1), we have the equality
    \begin{align*}
    (\gamma_* \Sym^m_{\sO_Q}\Omega^i_{Q/k})^G &= \gamma_* \Sym^m_{\sO_Q}\Omega^i_{Q/k} \cap (\Sym^m_{K(Y)}\Omega^i_{K(Y)/k})^G\\
    & = \gamma_* \Sym^m_{\sO_Q}\Omega^i_{Q/k} \cap \Sym^m_{K(X)}\Omega^i_{K(X)/k}.
    \end{align*}
    
    Since $(X,P)$ is plt, we have $P$ is regular \cite[Section 3.35]{Kol13}, and in particular, we can write $P=\Spec R$ and $Q=\Spec S$ for some discrete valuation rings $R$ and $S$. 
    Since $\gamma: Y \to X$ is \'etale in codimension one, we obtain
    \[
    p \not| \; [K(Y):K(X)] = [K(Q):K(P)].
    \]
    by applying \cite[\href{https://stacks.math.columbia.edu/tag/09E8}{Tag 09E8}]{stacks-project} to $\sO_{X,\eta_P} \to \sO_{Y, \eta_Q}$, and in particular,
    $R\to S$ is finite of degree prime to $p$.
    By applying \cite[\href{https://stacks.math.columbia.edu/tag/09E8}{Tag 09E8}]{stacks-project} to $R \to S$, the ramification index $e$ of $S$ over $R$ is equal to $[K(Y):K(X)]=m$.
    It follows from \cite[Claim 3.35.1 and Proposition 10.9 (3)]{Kol13} that we have $m = \det (\Gamma)$, where $\Gamma$ is the intersection matrix of a log minimal resolution of $(X,P)$.
    Combining this with the adjunction \cite[Theorem 3.36]{Kol13},
    we have 
    \[
    \mathrm{Diff}_P(0) = (1-\frac{1}{e})\{\m_R\}.
    \]
    Then the assertion (3) follows from Lemma \ref{lem:key inclusion in DVR case}.
\end{proof}

\begin{defn}
A pair $(X, B)$ over a field $k$ of characteristic $p>0$ is said to be \textit{tamely dlt} if $(X,B)$ is dlt, $B$ is reduced, and the Cartier index of $K_X+B$ is not divisible by $p$.
\end{defn}

\begin{thm}[\textup{cf.~\cite[Theorems 1.4 and 1.5]{Gra}}]\label{thm:reside and restr ex seq}
    Let $(X,P+D)$ be a two-dimensional tamely dlt pair over a perfect field $k$ of characteristic $p>0$.
    We assume that $k$ contains an $\ell$-th root of unity, where $\ell$ is the Cartier index of $K_X+P+D$.
    Let $i>0$ be a positive integer and $D_P$ be the different of $D$ on $P$.
    \begin{enumerate}[label=\textup{(\arabic*)}]
        \item There exists a short exact sequence
            \[
            0 \to \Omega^{[i]}_{X/k}(\log\,  D  )\to \Omega^{[i]}_{X/k}(\log D+P ))  \xrightarrow{\res^1_P} \Omega^{i-i}_{P/k}(\log \lfloor D_P\rfloor) \to 0\]
            and a surjective morphism 
            \[
            \res^m_{P}\colon \Sym^{[m]}_{\mathcal{C}}\Omega_{X/k}^{[i]}(\log D+P)\to \Sym^{m}_{\mathcal{C}} \Omega^{i-1}_{P/k}(\log D_P)
            \] 
            for each $m > 0$ that coincides with the map \eqref{map:res} at the generic point of $P$.
       \item There exists a short exact sequence
            \[
            0 \to (\Omega^{[i]}_{X/k}(\log D+P )\otimes\sO_X(-P))^{**} \to \Omega^{[i]}_{X/k}(\log  D ) \xrightarrow{\restr^1_P} \Omega^{i}_{P/k}(\log \lfloor D_P \rfloor) \to 0
            \]
            and a surjective morphism 
            \[
            \mathrm{restr}^{m}_{P}\colon \Sym^{[m]}_{\mathcal{C}} \Omega^{[i]}_{X/k}(\log D)\to \Sym^{m}_{\mathcal{C}} \Omega^{i}_{P/k}(\log D_P)
            \]
            for each $m>0$ that coincides with the map \eqref{map:restr} at the generic point of $P$.
\end{enumerate}
\end{thm}

\begin{proof}
     We only show (1) as the other is similar.
     Let $U \subseteq X$ be the maximal open subset where $(X,D+P)$ is snc.
     Then the push-forward of the residual map
     \[
     \res_{P \cap U}^m \colon \Sym^m_{\mathcal{C}}\Omega^i_{U/k}(\log (D+P)|_U)\onto \Sym^m_{\mathcal{C}} \Omega^{i-1}_{P \cap U/k}(\log D|_{U \cap P}),
     \]
     defines a morphism
     \[
     \res_{P}^m \colon \Sym^{[m]}_{\mathcal{C}}\Omega^{[i]}_{X/k}(\log D+P)\to \Sym^m_{K(P)/k} \Omega^{i-1}_{K(P)/k},
     \]
     where $\Sym^m_{K(P)/k} \Omega^{i-1}_{K(P)/k}$ is the constant sheaf on $P$.
     It is enough to show that the image of $\res_P^m$ is $\Sym^{m}_{\mathcal{C}} \Omega^{i-1}_{P/k}(\log D_P)$.
     
     We may assume that $X$ is the spectrum of a local ring and $P \neq 0$.
     Since the assertion holds on the snc locus $U$, we may also assume that $D=0$.
     Take $\gamma \colon Y \to X$ and $Q=\gamma^{-1} (P)_{\mathrm{red}}$ as in Lemma \ref{lem:quasi-etale cover}.
     We note that the following diagram commutes:
    \[
    \begin{tikzcd}
    \Sym^{[m]}_{\sO_X} \Omega_{X/k}^{[i]}(\log P)  \arrow[dd, "\res_P^m"] \arrow[r, hookrightarrow]  & \gamma_* \Sym^m_{\sO_Y} \Omega_{Y/k}^i(\log Q)  \arrow[d, ->>, "\gamma_*\res_Q^m"] \\
    & \gamma_* \Sym^m_{\sO_Q} \Omega_{Q/k}^{i-1} \arrow[d, hookrightarrow] \\
    \Sym^m_{K(P)} \Omega_{K(P)/k}^{i-1} \arrow[r, hookrightarrow] & \gamma_* \Sym^m_{K(Q)} \Omega_{K(Q)/k}^{i-1} .
    \end{tikzcd}
    \]
    Since $\res_P^m$ is a morphism of $G$-sheaves and the functor $(\gamma_* - )^G$ is exact, we also have a commutative diagram
    \[
    \begin{tikzcd}
    (\gamma_* \Sym^m_{\sO_Y} \Omega_{Y/k}^i(\log Q))^G  \arrow[d, ->>, "(\gamma_* \res_Q^m)^G"] \arrow[r, hookrightarrow]  & \gamma_* \Sym^m_{\sO_Y} \Omega_{Y/k}^i(\log Q)  \arrow[d, ->>, "\gamma_*\res_Q^m"] \\
    (\gamma_* \Sym^m_{\sO_Q} \Omega_{Q/k}^{i-1})^G \arrow[r, hookrightarrow] \arrow[d, hookrightarrow] & \gamma_* \Sym^m_{\sO_Q} \Omega_{Q/k}^{i-1} \arrow[d, hookrightarrow] \\
    (\gamma_* \Sym^m_{K(Q)} \Omega_{K(Q)/k}^{i-1})^G \arrow[r, hookrightarrow] & \gamma_* \Sym^m_{K(Q)} \Omega_{K(Q)/k}^{i-1} .
    \end{tikzcd}
    \]
    Combining this with Lemma \ref{lem:quasi-etale cover}, we have
    \[
    \Im(\res^m_P) = (\gamma_* \Sym^m_{\sO_Q} \Omega_{Q/k}^{i-1})^G = \Sym^m_{\mathcal{C}} \Omega^{i-1}_{P/k}(\log D_P),
    \]
    as desired.
\end{proof}

\subsection{Tame decomposition}
In this subsection, we give a sufficient condition for a two-dimensional pair to satisfy the logarithmic extension theorem (Lemma \ref{lem:tame decomposition implies LET}).

\begin{prop}\label{prop:LET for tame dlt pair}
    Let $(X=\Spec R,B)$ be a two-dimensional pair over a perfect field $k$ of characteristic $p>0$ with $(R,\m)$ local and $B$ reduced.
    Suppose that $f \colon Y \to X$ is a proper birational morphism from a normal irreducible scheme $Y$.
    Set $B_Y=f_{*}^{-1}B+\Exc(f)$.
    Suppose that 
    \begin{enumerate}
        \item[\textup{(1)}] $(Y,B_Y)$ is tamely dlt,
        \item[\textup{(2)}] $-(K_Y+B_Y)$ is $f$-nef, 
        \item[\textup{(3)}] every stratum of $B_Y$ is geometrically regular over $\kappa \coloneqq R/\m$, and
        \item[\textup{(4)}] the residue field at every point in the non-snc locus of $(Y, B_Y)$ is separable over $\kappa$.
    \end{enumerate}
    Then a natural restriction map
    \[
    f_{*}\Omega^{[i]}_{Y/k}(\log B_Y)\hookrightarrow\Omega^{[i]}_{X/k}(\log B)
    \]
    is surjective for all $i\geq 0$.
\end{prop}
\begin{proof}
After replacing $k$ by its finite separable extension, we may assume that $k$ contains an $\ell$-th root of unity, where $\ell$ is the Cartier index of $K_Y+B_Y$ (cf.~\cite[Proposition 7.4]{Gra}).
We also note that $Y$ is $\Q$-factorial since $(Y, B_Y)$ is dlt \cite[Corollary 4.11]{Tanaka(exc)}.

When $i=0$, the desired surjectivity is nothing but $f_{*}\sO_Y=\sO_X$.
Fix $i\geq 1$. We show that a natural restriction map
\[
H^0(Y, \Omega_{Y/k}^{[i]}(\log B_Y))\hookrightarrow H^0(X, \Omega_{X/k}^{[i]}(\log B))
\]
is surjective.

We take $s\in H^0(X, \Omega_{X/k}^{[i]}(\log B))$. 
Then we can take an effective $f$-exceptional divisor $F$ and a section
$s_Y\in H^0(Y, \Omega_{Y/k}^{[i]}(\log B_Y)\otimes\sO_Y(F))$ that maps to $s$ by restriction.
Our aim is to show that we can take $F=0$.

The section $s_Y$ can be considered as a map
\[
s_Y\colon \sO_X(-F) \hookrightarrow \Omega^{[i]}_{Y/k}(\log B_Y).
\]
Take a prime $f$-exceptional divisor $E_1$ such that $E_1\subset \Supp(F)$ and $F\cdot E_1<0$.
In what follows,
we prove that $s_Y$ factors though an injective map
\[
\sO_{Y}(-F+E_1)\hookrightarrow \Omega^{[1]}_{Y/k}(\log\,B_Y).
\]

Since $(Y, B_Y)$ is tamely dlt, we can use Theorem \ref{thm:reside and restr ex seq} (1) to obtain the following commutative diagram 
\begin{equation*}
\xymatrix{ & & \sO_{Y}(-F) \ar@{.>}[ld] \ar[d]^-{s_Y} \ar[rd]^-{t} &&\\
                 0\ar[r] &\Omega^{[i]}_{Y/k}(\log B_Y-E_1)\ar[r]   & \Omega^{[i]}_{Y/k}(\log\,B_Y) \ar[r]^-{\res_{E_1}}  & \Omega^{i-1}_{E_1/k}(\log \lfloor B_{E_1}\rfloor) \to 0,}
\end{equation*}
and a morphism 
\[
\res^m_{E_1}\colon \Sym^{[m]}_{\mathcal{C}}\Omega_{Y/k}^{[i]}(\log\,B_Y)\to \Sym^{m}_{\mathcal{C}}\Omega^{i-1}_{E_1/k}(\log B_{E_1})
\]
for every integer $m>0$, where $B_{E_1}$ is the different of $B_Y-E_1$ on $E_1$.

We show that $t$ is zero. Suppose by contradiction that $t$ is non-zero. As $\Im(t)\subset \Omega^{i-1}_{E_1/k}(\log\,\lfloor B_{E_1}\rfloor)$ is a torsion-free $\sO_{E_1}$-module, it follows that $t$ is non-zero at the generic point $\eta_1$ of $E_1$ and 
\[
\Sym^{m}(t)\colon \Sym^m_{\sO_Y} \sO_Y(-F) \to 
\Sym^m_{\sO_{E_1}}\Omega^{i-1}_{E_1/k}(\log \lfloor B_{E_1}\rfloor)
\]
is also non-zero at $\eta_1$. 
By the construction of $\res^m_{E_1}$, the map $\Sym^m(\res_{E_1})$ coincides with $\res^m_{E_1}$ at $\eta_1$.
Moreover, $\Sym^{m}(s_Y)$ coincides with $\Sym^{[m]}(s_Y)\coloneqq (\Sym^m(s_Y))^{**}$ at $\eta_1$ as $Y$ is smooth at $\eta_1$.
Thus, the composition $\res^m_{E_1}\circ\Sym^{[m]}(s_Y)$ coincides with $\Sym^m(t)=\Sym^m(\res_{E_1})\circ\Sym^m(s_Y)$, and in particular, it is non-zero at $\eta_1$. 
Fix $m>0$ such that $mD_Y$ is Cartier. 
By restricting $\res^m_{E_1}\circ\Sym^{[m]}(s_Y)$ to $E_1$, we obtain an injective map \[\sO_{E_1}(-mF)\hookrightarrow \Sym^{m}_{\mathcal{C}}\Omega^{i-1}_{E_1/k}(\log B_{E_1}).\]
Since $(-mF)\cdot E_1=\deg(\sO_{E_1}(-mF))>0$, this contradicts the following claim.

\begin{cl}
    $H^0(E_1, \Sym^{m}_{\mathcal{C}} \Omega^{i}_{E_1/k}(\log  B_{E_1} )\otimes\sO_{E_1}(-G))=0$ for all $i\geq 0$, $m\geq 0$, and an ample Cartier divisor $G$ on $E_1$.
\end{cl}

\begin{proof}
    We set $E : =E_1$ and $D : =B_{E_1}$.
    Since every point $x \in \Supp(D)$ is either a zero dimensional stratum of $B_Y$ or a point contained in the non-snc locus of $(Y,B_Y)$, it follows from the assumption that the residue field $\kappa(x)$ is separable over $\kappa$.
    Therefore, $(E,D)$ is a pair over $\kappa$ with geometrically snc support.
    Let $F^{p,l}_{i,m}$ be a locally free $\sO_{E}$-submodule of $\Sym^m_{\mathcal{C}} \Omega^i_{E/k}(\log D)$ in Lemma \ref{lem:change of fields}.

    We fix $i \ge 0$ and we verify the vanishing 
    \begin{align}\label{eq:vanishing for C-diff}
    H^0(E, F^{p,l}_{i,m}(-G))=0.
    \end{align}
    for every $p,m,l \ge 0$ and for every ample divisor $G$ on $E$.
    
    \textbf{Case 1}:
    If $p \ge i+1$, then this follows from the fact that 
    \[
    F^{p,l}_{i,m}=\left\{ \begin{array}{ll} \sO_E & \textup{(if $m=l=0$)} \\ 0 & \textup{(otherwise)} \end{array}\right.
    \]
    (Lemma \ref{lem:change of fields} (iii)).

    \textbf{Case 2}:
    We next consider the case of $p =i$.
    By Lemma \ref{lem:change of fields} (ii), we have the short exact sequence
    \[
    0 \to F^{p,l+1}_{i,m} \to F^{p,l}_{i,m} \to F^{p+1,0}_{i,l} \otimes_{\kappa} \Sym^{m-l}_{\mathcal{C}} (\Omega^p_{\kappa/k} \otimes_{\kappa} \Omega^{i-p}_{E/\kappa}(\log D)) \to 0.
    \]   
    By Remark \ref{rem:C-differential} (3), 
    the module $\Sym^{m-l}_{\mathcal{C}}(\Omega^p_{\kappa/k} \otimes_{\kappa} \Omega^{i-p}_{E/\kappa}(\log D))$ is a free $\sO_E$-module.
    Combining this with Case 1, we have 
    \[
    H^0(F^{p,0}_{i,m}(-G)) = H^0(F^{p,1}_{i,m}(-G))= \cdots =H^0(F^{p,m+1}_{i,m}(-G)).
    \]
    Then these global sections are zero since $F^{p,m+1}_{i,m}=0$ by Lemma \ref{lem:change of fields} (iii).

    \textbf{Case 3}:
    In the case of $p=i-1$, the equation \eqref{eq:vanishing for C-diff} follows from the similar argument as in Case 2, by using
    \[
    \Sym^{m-l}_{\mathcal{C}}(\Omega^p_{\kappa/k} \otimes_{\kappa} \Omega^{i-p}_{E/\kappa}(\log D)) \cong \Sym^{m-l}_{\kappa} \Omega^p_{\kappa/k} \otimes_{\kappa} \sO_E(\lfloor (m-l)(K_E+D) \rfloor)
    \]
    (Remark \ref{rem:C-differential} (3)) and the anti-nefness of $K_E+D$.

    \textbf{Case 4}:
    We finally consider the case of $p \le i-2$.
    By Remark \ref{rem:C-differential} (3) again, we have
    \[
    \Sym^{m-l}_{\mathcal{C}}(\Omega^p_{\kappa/k} \otimes_{\kappa} \Omega^{i-p}_{E/\kappa}(\log D)) \cong \left\{ \begin{array}{ll} \sO_E & \textup{(if $m=l$) } \\
    0 & \textup{(otherwise) } 
    \end{array}\right..
    \]
    Combining this with Lemma \ref{lem:change of fields} (i) and (ii), we have
    \[
    F^{p,0}_{i,m}= F^{p,1}_{i,m} = \dots =F^{p,m}_{i,m} \cong F^{p+1,0}_{i,l}.
    \]
    Therefore, the desired vanishing follows from induction on $p$.
    \end{proof}

Therefore, $t$ is zero and the morphism $s_Y$ factors through \[\sO_{Y}(-F) \to \Omega^{[1]}_{Y/k}(\log\,B_Y-E_1).\]
Then, by Theorem \ref{thm:reside and restr ex seq} (2), we obtain the following commutative diagram
\begin{small}
\begin{equation*}
\xymatrix{  & \sO_{Y}(-F) \ar@{.>}[ld] \ar[d]^{s_Y} \ar[rd]^{v} &\\
                 0 \to (\Omega^{[i]}_{Y/k}(\log\,B_Y)\otimes\sO_Y(-E_1))^{**} \ar[r]   & \Omega^{[i]}_{Y/k}(\log\,B_Y-E_1) \ar[r]^-{\mathrm{restr_{E_1}}}  & \Omega^{[i]}_{E_1/k}(\log\,\lfloor B_{E_1} \rfloor) \to 0,}
\end{equation*}
\end{small}
and a morphism 
\[
\mathrm{restr}^{m}_{E_1}\colon \Sym^{[m]}_{\mathcal{C}}\Omega^{[i]}_{Y/k}(\log\,(B_Y-E_1))\to \Sym^{[m]}_{\mathcal{C}}\Omega^{[i]}_{E_1/k}(\log B_{E_1} )
\]
that coincides with $\Sym^{m}(\mathrm{restr_{E_1}})$ at the generic point $E_1$.

Therefore an argument similar to above shows that $v=0$ and $s_Y$ factors through $\sO_{Y}(-F+E_1) \hookrightarrow \Omega^{[1]}_{Y/k}(\log B_Y)$.
By replacing $-F+E_1$ with $-F$, and repeating the above procedure, we can prove that we can take $F=0$.
\end{proof}

\begin{defn}\label{def:tame decomposition}
    Let $(X=\Spec R,B)$ be a two-dimensional pair over a perfect field $k$ of characteristic $p>0$ with $(R,\m)$ local and $B$ reduced.
    Suppose that $f\colon Y \to X$ is log resolution and we set $B_Y \coloneqq f_{*}^{-1}B+\Exc(f)$.
    \textit{A tame decomposition of $f$} is a sequence
    \[
    \begin{tikzcd}[column sep=1cm]
    Y=Y_0 \arrow[rrrr,bend right, "f"] \arrow[r, "\phi_0"] & Y_1 \arrow[r, "\phi_1"] & \cdots \arrow[r, "\phi_{r-2}"] & Y_{r-1} \arrow[r, "\phi_{r-1}"] & Y_r =X
    \end{tikzcd}
    \]
     of proper birational morphisms
    such that for any $0 \leq i \leq r-1$, setting $B_{Y_i}\coloneqq (\phi_{i-1}\circ\cdots\circ\phi_0)_{*}B_Y$, the following properties hold:
    \begin{enumerate}[label=\textup{(\arabic*)}]
    \item every stratum of $B_Y$ is smooth over $R/\m$,
    \item the pair $(Y_i, B_{Y_i})$ is tamely dlt for every $i$, and 
    \item $-(K_{Y_i} +B_{Y_i})$ is nef over $Y_{i+1}$ for every $i$.
    \end{enumerate}
\end{defn}

\begin{lem}\label{lem:tame decomposition implies LET}
    Let $(X,B)$ be as in Definition \ref{def:tame decomposition}.
    If there exists a log resolution $f\colon Y \to X$ which admits a tame decomposition, then $X$ satisfies the logarithmic extension theorem for differential forms of any degree.
\end{lem}
\begin{proof}
    Take a tame decomposition of $f$ as in Definition \ref{def:tame decomposition}.
    It suffices to show that for every $i$ and $j$, we have
    \[
    (\phi_j)_{*}\Omega^{[i]}_{{Y_j}/k}(\log B_{Y_j}) = \Omega^{[i]}_{{Y_{j+1}}/k}(\log B_{Y_{j+1}}).
    \]
    Let $Z \subseteq Y_{j}$ be a stratum of $\Exc(\phi_j)$ or a singleton set consisting of a closed point in the non-snc locus of $(Y_j, B_{Y_j})$, and $y \in Y_{j+1}$ be the closed point with $\phi_j(Z) = \{y\}$.
    By Proposition \ref{prop:LET for tame dlt pair}, it suffices to show that $Z$ is smooth over the residue field $\kappa(y)$ of $Y_{j+1}$ at $y$.

    When $Z$ is a stratum of $\Exc(\phi_j)$, we can take a stratum $W \subseteq Y$ of $\Exc(f)$ which maps onto $Z$.
    If $Z=\{x\}$ for some non-snc locus $x \in Y_j$, then we may find an irreducible component $W$ of $\Exc(f)$ which maps onto $Z$.
    Therefore, in both cases, we obtain the decomposition
    \[
    W \onto Z \to \Spec \kappa(y) \to \Spec \kappa
    \]
    of the smooth morphism $W \to \Spec \kappa$, where $\kappa=R/\m$.
    It follows from the descent of smoothness \cite[\href{https://stacks.math.columbia.edu/tag/05B5}{Tag 05B5}]{stacks-project} that $\Spec \kappa(y)$ is smooth over $\Spec \kappa$.
    Combining this with the fact that $\kappa(y)$ is of finite type over $\kappa$, it is a finite separable extension.
    Since we have $\Omega_{\kappa(y)/\kappa}^1=0$, the first exact sequence shows that $\Omega_{W/\kappa(y)}^1 \cong \Omega_{W/\kappa}^1$. Thus, $\Omega_{W/\kappa(y)}^1$ is locally free of rank $\dim W$, and $W$ is smooth over $\kappa(y)$.
    By applying the descent of smoothness again, we conclude that $Z$ is smooth over $\kappa(y)$.
\end{proof}

\subsection{Proof of Theorem \ref{Introthm:2-dim}: strongly \texorpdfstring{$F$}{F}-regular case}
In this subsection, we use the classification of strongly $F$-regular surface singularities which was given in \cite[Theorem A]{Kawakami-Sato2}.
Note that strongly $F$-regular surface singularities are klt, and the list of dual graphs of the minimal resolutions of klt surface singularities can be found in \cite[Figure 1]{Sato23}.

\begin{thm}\label{mainthm:2-dim}
    Let $X$ be a two-dimensional normal irreducible scheme essentially of finite type over a perfect field $k$ of characteristic $p>0$.
    If $X$ is strongly $F$-regular, then $X$ satisfies the logarithmic extension theorem for differential forms of any degree.
\end{thm}
\begin{proof}
We may assume that $X=\Spec R$ for a local ring $(R,\m,\kappa)$.
Let $\pi\colon Y \to X$ be the minimal resolution.
Since $X$ is strongly $F$-regular, it follows from \cite[Theorem A]{Kawakami-Sato2} that every exceptional curve $E_i$ is geometrically reduced over $\kappa$.
Since the logarithmic extension theorem can be checked \'etale locally (cf.~\cite[Lemma 4.6]{Heuver}), applying \cite[Lemma 4.3]{Sato23}, we may assume that every $E_i$ is geometrically irreducible.
Then the geometric integrality of $E_i$ implies that we have $H^0(E_i, \sO_{E_i})=\kappa$, and thus the dual graph of $\pi$ is either a chain or star shaped (cf.~\cite[Figure 1]{Sato23}).

By Lemma \ref{lem:tame decomposition implies LET}, it suffices to show that $\pi$ admits a tame decomposition.
We first note that $E_i$ is smooth over $\kappa$ for every $i$ \cite[Lemma 3.1]{Sato23}.
By the classification of dual graphs, every zero-dimensional stratum of $\Exc(\pi)$ is isomorphic to $\Spec \kappa$.
Therefore, $\pi$ is a log resolution and the condition (1) of Definition \ref{def:tame decomposition} is satisfied.
The rest of the proof is similar to that of \cite[Theorem 1.2]{Gra}.
But we give a proof for the reader’s convenience.
        \begin{enumerate}
        \setcounter{enumi}{0}
        \item (Chain case. (cf.~\cite[Figure 5]{Gra}))
        In this case, we will show that 
        \[
        Y=Y_0 \xrightarrow{\pi} Y_1 =X
        \]
        satisfies the conditions in Definition \ref{def:tame decomposition} since  $-(K_Y+E)$ is nef over $X$ as follows:
        \begin{align*}
        (K_Y+\Exc(\pi)) \cdot E_i &= (K_Y+E_i) \cdot E_i + (\Exc(\pi) -E_i) \cdot E_i\\
        &= -2 + (\Exc(\pi)-E_i) \cdot E_i\leq 0.
        \end{align*}
        
        \item (Star shaped case of type $(2,3,3)$, $(2,3,4)$ or $(2,3,5)$. (cf.~\cite[Figure 7]{Gra}))
        Let $C \subseteq Y$ be the exceptional prime divisor located on the center of the tree.
        As in the proof of \cite[Corollary 1.4.3]{BCHM} (see also \cite[Lemma 6.7]{Bir}, using MMP \cite[Theorem 4.4]{Tanaka(exc)}, we can take an extraction of $C$.
        That is, there is a decomposition 
        \[
        \begin{tikzcd}[column sep=1cm]
        Y=Y_0 \arrow[rr,bend right, "\pi" '] \arrow[r, "\phi_0"] & Y_1 \arrow[r, "\phi_1"] &  Y_2 =X,
        \end{tikzcd}
        \]
        where the exceptional locus of $\phi_1$ is $C_1 \coloneqq (\phi_0)_*C$.
        We will show that this decomposition satisfies the conditions in Definition \ref{def:tame decomposition}.
        
        The nefness of $-(K_{Y}+\Exc(f))$ over $Y_1$ follows from a direct computation as in (1). 
        Moreover, $-(K_{Y_1}+C_1)$ is nef over $X$ since $-C_1$ is ample over $X$ \cite[Lemma 10.2]{Kol} and $K_{Y_1}+a C_1=\phi_1^* K_X$ for some $a<1$.
        
        We finally show that the pair $(Y_1, C_1)$ is tamely dlt.
        Noting that $K_Y+C$ is nef over $Y_1$, the morphism $\phi_0\colon Y_0 \to Y_1$ is a log minimal resolution \cite[Theorem 2.25 (a)]{Kol13} of $(Y_1, C_1)$.
        Therefore, the extended dual graph of the pair $(Y_1, C_1)$ at each non-snc locus is a part of the tree:
                \[
\begin{tikzpicture}
\node[draw, shape=circle, inner sep=1.8pt] (E1) at (0,0){$*$};
\draw node (Dot1) at (1,0){$\cdots$};
\node[draw, shape=circle, inner sep=1.8pt] (C1) at (2,0){$*$};
\node[draw, shape=circle, inner sep=1.8pt, fill=black] (C) at (3,0){$*$};

\draw[shift={(0,0.4)}] (3,0) node{\tiny $C$};

\draw (E1)--(Dot1);
\draw (Dot1)--(C1);
\draw (C1)--(C);
\end{tikzpicture}
        \]
        It then follows from \cite[Claim 3.35.1]{Kol13} that $(Y_1,C_1)$ is plt.
        To verify the tameness, we consider only the case where the dual graph of $\pi$ is star shaped of type $(2,3,5)$ as the others are similar.
        In this case, we have $p>5$ by \cite[Theorem A]{Kawakami-Sato2}.
        On the other hand, since $\phi_0$ is a minimal resolution of $Y_1$, by \cite[Proposition 10.9 (3)]{Kol13}, we obtain that the Cartier index of $K_{Y_1}+C_1$ divides 30.
        \item (Star shaped case of type $(2,2,d)$ (cf.~\cite[Figure 6]{Gra}))
        Let $C_1,C_2 \subseteq Y$ be $(-2)$-curves located on the shortest branches of the dual graph.
        As in the argument in (2), we obtain a decomposition 
        \[
        \begin{tikzcd}[column sep=1cm]
        Y=Y_0 \arrow[rr,bend right, "\pi" '] \arrow[r, "\phi_0"] & Y_1 \arrow[r, "\phi_1"] &  Y_2 =X
        \end{tikzcd}
        \]
        with $\Exc(\phi_0)=C_1 \cup C_2$.
       By the similar argument to (2), the pair $(Y_1,\Exc(\phi_1))$ is dlt with the Cartier index two.
        Since we have $p \neq 2$ (\cite[Theorem A]{Kawakami-Sato2}), this pair is tamely dlt.
        The nefness of $-(K_Y+\Exc(\pi))$ over $Y_1$ is obvious.
        In order to prove that $-(K_Y+\Exc(\phi_1))$ is nef over $X$, we take a $\phi_1$-exceptional curve $C \subseteq Y_1$.
        Then we have
        \begin{align}\label{eq:intersection number}
        (K_Y+\Exc(\phi_1))\cdot C &=(K_Y+C) \cdot C + (\Exc(\phi_1)-C) \cdot C\notag \\
        &=\deg_C (K_C + \mathrm{Diff}_C(0)) +(\Exc(\phi_1)-C) \cdot C \notag\\
        &=-2 + \deg_C(\mathrm{Diff}_C(0)) +(\Exc(\phi_1)-C) \cdot C. 
        \end{align}
        If the strict transform $\widetilde{C}$ on $Y$ is not the center of the tree, then \eqref{eq:intersection number} is negative since 
        \[
        \mathrm{Diff}_C(0) = 0 \textup{  and } (\Exc(\phi_1)-C) \cdot C \leq2.
        \]
        When $\widetilde{C}$ is the center of the tree, the negativity of \eqref{eq:intersection number} follows from the equation
        \[
        \mathrm{Diff}_{C}(0) = \frac{1}{2}P_1 + \frac{1}{2}P_2,
        \]
        where $P_i$ is the center of $C_i$ \cite[Theorem 3.36]{Kol}.
 \end{enumerate}\qedhere
\end{proof}

\subsection{Proof of Theorem \ref{Introthm:2-dim}: log canonical case}

\begin{thm}\label{mainthm:2-dim(lc)}
    Let $(X,B)$ be a two-dimensional pair over a perfect field $k$ of characteristic $p>0$ with $B$ reduced.
    Suppose that $(X,B)$ is log canonical and $p>5$.
    Then $(X,B)$ satisfies the logarithmic extension theorem for differential forms of any degree.
\end{thm}
\begin{proof}
If $(X,B)$ is klt, then we have $B=0$ and $X$ is strongly $F$-regular (\cite[Theorem A]{Kawakami-Sato2} or \cite[Theorem 1.2]{Sato-Takagi(generalhyperplane)}).
In particular, the assertion follows from Theorem \ref{mainthm:2-dim}.
From now on, we assume that $(X,B)$ is not klt.
We may also assume that $X=\Spec R$ for a local ring $(R,\m,\kappa)$.

Let $\pi\colon Y \to X$ be the minimal resolution.
We first show that every exceptional cuver $E_i$ is geometrically reduced over $\kappa$. 
When $B=0$, this follows from the proof of \cite[Proposition 4.4, 4.8, 4.10, 4.11, 4.13 and 4.14]{Sato23}.
If $B \neq 0$, it follows from the classification \cite[Section 3.35]{Kol13} that any exceptional curve $E_i$ has $[H^0(\sO_{E_i}): R/\m] \leq2$ and genus zero.
Then we obtain the geometrical reducedness by \cite[Lemma 3.3]{Sato23}.

On the other hand, since the logarithmic extension theorem can be checked \'etale locally (cf.~\cite[Lemma 4.6]{Heuver}), applying \cite[Lemma 4.3]{Sato23}, we may assume that every $E_i$ is geometrically irreducible.
Then the geometric integrality of $E_i$ implies that we have $H^0(E_i, \sO_{E_i})=\kappa$, and thus the dual graph of $\pi$ is either a chain or star shaped (cf.~\cite[Figure 2 and 4]{Sato23}).

\textbf{Case 1:} We first consider the case where the singularity of $X$ is not of simple elliptic.
By Lemma \ref{lem:tame decomposition implies LET}, it suffices to show that $\pi$ admits a tame decomposition.
We first note that $E_i$ is smooth over $\kappa$ for every $i$ \cite[Lemma 3.1]{Sato23}.
By the classification of dual graphs, every zero dimensional stratum of $\Exc(\pi)$ is isomorphic to $\Spec \kappa$.
Therefore, $\pi$ is a log resolution and the condition (1) of Definition \ref{def:tame decomposition} is satisfied.
We write
\[
K_Y+ \Delta_Y = \pi^*(K_X+B).
\]
As in \cite[Lemma 6.7]{Bir}, using MMP \cite[Theorem 4.4]{Tanaka(exc)}, we can extract all exceptional curves $E_i$ with $\mathrm{ord}_{E_i}(\Delta_Y)=1$.
    That is, there is a decomposition 
        \[
        \begin{tikzcd}[column sep=1cm]
        Y=Y_0 \arrow[rr,bend right, "\pi" '] \arrow[r, "\phi_0"] & Y_1 \arrow[r, "\phi_1"] &  Y_2 =X,
        \end{tikzcd}
        \]
        where the exceptional locus of $\phi_1$ is the union of all exceptional curves $E_i$ with $\mathrm{ord}_{E_i}(\Delta_Y)=1$.
        Then this decomposition satisfies the conditions in Definition \ref{def:tame decomposition} by a similar argument to \cite[Theorem 1.2]{Gra} or Theorem \ref{mainthm:2-dim}.

    \textbf{Case 2:}
    In what follows, we treat the case where $X$ is simple elliptic.
   In this case, we have $B=0$ and $E=\Exc(\pi)$ consists of a single curve with genus one.
        We have two subcases: either $E$ is regular or singular.
        In the former case, the minimal resolution $\pi$ is a log resolution.
        Since we have $\pi^*K_X = K_Y+E$, the morphism
        \[
        Y= : Y_0 \xrightarrow{\pi} Y_1 \coloneqq X
        \]
        satisfies the conditions (2) and (3) in Definition \ref{def:tame decomposition}.
        The condition (1) follows from  \cite[Theorem 1.1]{PW} (cf.~\cite[Proposition 9.11 (2)]{Tanaka(invariants)}) as $p>3$.
        Thus, $\pi$ admit a tame decomposition.
        
        In the latter case, by the proof of \cite[Proposition 4.13]{Sato23}, $E$ has a unique singular point $P \in E$.
        Moreover, $E$ is smooth over $\kappa$ outside $P$, $E$ is node at $P$, and the residue field $\kappa(P)$ is isomorphic to $\kappa$.
        
        The composite $f: Z \to Y$ of $\pi$ and the blowing up $Z \to Y$ at $P$ is a log resolution of $X$.
        The strict transform of $E$ to $Z$ is smooth over $\kappa$ by \cite[Lemma 3.1]{Sato23}, which implies that each strata of $\Exc(f)$ is smooth over $\kappa$.
        Combining this with the equation $f^*K_X = K_Z+\Exc(f)$, we conclude that $f$ admits a tame decomposition
        \[
        Z=: Y_0 \xrightarrow{f} Y_1 \coloneqq X,
        \]
        and we conclude.
\end{proof}

As an immediate consequence of Theorems \ref{mainthm:2-dim} and \ref{mainthm:2-dim(lc)}, we obtain the following assertion:
\begin{cor}\label{cor:log ext for codim two}
    Let $X$ be a normal variety over a perfect field of characteristic $p>0$.
    Suppose that one of the following holds:
    \begin{enumerate}[label=\textup{(\arabic*)}]
        \item $X$ is strongly $F$-regular.
        \item $X$ is lc and $p>5$.
    \end{enumerate}
    Then, for any proper birational morphism $f\colon Y\to X$ with reduced $f$-exceptional divisor $E$, the restriction map
    \[
    f_{*}\Omega^{[i]}_Y(\log E)\to f_{*}\Omega^{[i]}_X
    \]
    is surjective in codimension two for all $i\geq 0$.
\end{cor}

\section*{Acknowledgements}
We wish to express our gratitude to  Kenta Hashizume, Hiroki Matsui, Austyn Simpson, and Shou Yoshikawa
for valuable conversations.
Kawakami was supported by JSPS KAKENHI Grant number
JP22KJ1771 and JP24K16897.
Sato was supported by JSPS KAKENHI Grant number
JP20K14303 and JP24K16900.

\bibliography{KawakamiSato_LET_arXiv_v3.bib}
\bibliographystyle{alpha}

\bigskip

\end{document}